\documentclass[10pt,a4paper]{article}
\usepackage[utf8]{inputenc}
\usepackage[T1]{fontenc}
\usepackage[english]{babel}

\usepackage[top=0.75in, bottom=0.75in, left=0.75in, right=0.75in]{geometry}

\usepackage{amsmath,amssymb,enumerate,latexsym,xcolor,amsthm,mathtools,microtype,authblk}

\usepackage{newtxtext}
\usepackage[smallerops]{newtxmath}
\usepackage{calrsfs}

\numberwithin{equation}{section} 
\allowdisplaybreaks 
\usepackage{geometry}
\usepackage[colorlinks=true, linkcolor=black, citecolor=black, urlcolor=blue!60!black]{hyperref}

\newcommand{\assign}{\coloneqq}

\newcommand{\equallim}{\mathop{=}\limits}
\newcommand{\mathd}{\mathop{}\!\mathrm{d}}

\newcommand{\tmop}[1]{\ensuremath{\operatorname{#1}}}

\newcommand{\tmtextit}[1]{{\itshape{#1}}}
\newenvironment{enumeratealpha}{\begin{enumerate}[{\textup{(}}a{\textup{)}}] }{\end{enumerate}}
\newenvironment{enumerateroman}{\begin{enumerate}[{\textup{(}}i{\textup{)}}] }{\end{enumerate}}
\newtheorem{theorem}{Theorem}[section]
\newtheorem{corollary}[theorem]{Corollary}
\newtheorem{definition}[theorem]{Definition}
\newtheorem{lemma}[theorem]{Lemma}
\theoremstyle{remark}
\newtheorem{remark}[theorem]{Remark}
\theoremstyle{definition}
\newtheorem{assumption}{Assumption} 

\begin{document}

\title{ \bfseries 
  Noether theorem in stochastic optimal control\\
  problems via contact symmetries
}

\date{} 

\author[1]{Francesco C.~De Vecchi\thanks{\href{mailto:francesco.devecchi@uni-bonn.de}{francesco.devecchi@uni-bonn.de}}}

\author[2]{Elisa Mastrogiacomo\thanks{\href{mailto:elisa.mastrogiacomo@uninsubria.it}{elisa.mastrogiacomo@uninsubria.it}}}

\author[1]{\\Mattia Turra\thanks{\href{mailto:mattia.turra@iam.uni-bonn.de}{mattia.turra@iam.uni-bonn.de}}}

\author[3]{Stefania~Ugolini\thanks{\href{mailto:stefania.ugolini@unimi.it}{stefania.ugolini@unimi.it}}} 

\affil[1]{\normalsize Institute for Applied Mathematics and
	  Hausdorff Center for Mathematics,
	  University of Bonn, Germany}

\affil[2]{Dipartimento di Economia,
	  Universit{\`a} degli Studi dell'Insubria, Italy}

\affil[3]{Dipartimento di Matematica,
	  Universit{\`a} degli Studi di Milano, Italy}

\maketitle

\begin{abstract}
	We establish a generalization of Noether theorem for stochastic optimal control problems. 
	Exploiting the tools of jet bundles and contact geometry, we prove that from any (contact) symmetry of the Hamilton-Jacobi-Bellman equation associated to an optimal control problem it is possible to build a related local martingale.
	Moreover, we provide an application of the theoretical results to Merton's optimal portfolio problem, showing that this model admits infinitely many conserved quantities in the form of local martingales. 
	
	\medskip
	
	
	\medskip
	
\end{abstract}

\section{Introduction}

The concept of symmetry of ordinary or partial differential equations (ODEs and PDEs) was introduced by Sophus Lie at the end of the 19th century with the aim of extending the Galois theory from polynomial to differential equations. 
Actually, all the theory of Lie groups and algebras was developed by Lie himself as well as the principal tools for facing the problem of symmetries of differential equations (see~\cite{hawkins_emergence_2012} for an historical introduction to the subject and~\cite{olver_applications_1993,stephani_differential_1989} for some modern presentations).

\medskip

One of the most important application of the study of symmetries in physical systems was provided by Emmy Noether. 
She understood that when an equation comes from a variational problem, such as in Lagrangian mechanics, general relativity or, more generally, field theory, it is possible to relate each symmetry of the equation to a conserved quantity, i.e.,~a function of the state of the system that does not change during the evolution of the dynamics, and conversely, to each conserved quantity it is possible to associate a symmetry of the motion.
The simplest examples are, in Newtonian and Lagrangian mechanics, the conservation of energy, that is related to the invariance with respect to time translation, and the conservation of angular momentum, which is correlated to the invariance with respect to rotations (see, e.g.,~\cite{arnold_mathematical_1989,olver_applications_1993}).

\medskip

The development of a Lie symmetry analysis for stochastic differential equations (SDEs) and general random systems is relatively recent (see, e.g.,~\cite{ADMUweak,ADMUrandom,DMUanote,DMUgirsanovreduction,DMUgirsanov,DRUnumerical,Gaeta2019,Gaeta2017,Gaeta2020,Kozlov2018,Liao2019} for some recent developments in the non-variational case). 
For stochastic systems arising from a variational framework, it is certainly interesting to study the relation between their symmetries and functionals which are conserved by their flow, and, in particular, to establish stochastic generalizations of Noether theorem.\\
The problem of finding some kinds of conservation laws for SDEs was discussed in various papers (see~\cite{arnaudon_stochastic_2017,Baez2013,Luo2018,lescot_isovectors_2004,Lescot_Zambrini2008,Misawa1994,Misawa19942,privault_stochastic_2010,thieullen_symmetries_1997,zambrini_geometry_2009}). 
We could summarize three different approaches to this problem. 
The first one was considered by Misawa in~\cite{Misawa1994,Misawa19942,Misawa1999}, where the author studied the case in which some Markovian functions of solutions of SDEs are exactly conserved during time evolution. \\
The second approach was adopted by Zambrini and co-authors in a number of works. 
They put themselves in the framework of Euclidean quantum mechanics, which represents a geometrically consistent stochastic deformation of classical mechanics where a Gaussian noise is added to a classical system. 
This setting has a close connection with optimal transport and optimal control (see, e.g.,~\cite{Zambrini2015} for an introduction to the topic). 
More precisely, in~\cite{thieullen_symmetries_1997} a generalization of Noether theorem has been proved: to any one-parameter symmetry of a variational problem it is possible to associate a martingale which is independent both from the initial and final condition of the system. This first step was quite important since it stressed that the suitable generalization of conserved quantities in a stochastic setting is not a function that remains constant during the time evolution of a stochastic system, but a function that is constant in mean. 
Another remarkable advance in the study of variational symmetries was achieved in~\cite{lescot_isovectors_2004,Lescot_Zambrini2008,zambrini_geometry_2009}, where it was noted that the symmetries of the Hamilton-Jacobi-Bellman (HJB) equation of the considered variational problem are the correct objects to be associated to the aforementioned martingales and the contact geometry is a good framework in which a stochastic version of Noether theorem can be formulated. 
Indeed, to each Lie point symmetry of the HJB equation it is possible to associate a martingale for the evolution of the system. 
It is worth also mentioning the papers~\cite{arnaudon_stochastic_2017,privault_stochastic_2010}, where a suitable notion of integrable system, i.e.,~a system with a number of martingales and symmetries equal to the number of the dimension, is discussed.\\
The third approach was proposed by Baez and Fong in~\cite{Baez2013} (see also~\cite{Luo2018}). 
The authors showed a method to build martingales applying the action of symmetries to solution to backward Kolmogorov equation, that can be interpreted as a linear version of HJB equation obtained when the control and the objective function are trivial. 

\medskip 

In our paper, we generalize at least along two directions the approach proposed by Zambrini and co-authors, as listed above.
First, we work in a different optimal control setting that can be seen as a generalization of the variational framework described in their articles.
Second, we do not only restrict to Lie point symmetries but we take advantage of the general notion of contact symmetry, namely a transformation preserving the contact structure of the jet space (see Section~\ref{s:sym}). \\
We prove here a Noether theorem (Theorem~\ref{theorem_noether}) which relates to any contact symmetry of the HJB equation associated with an optimal control problem, a martingale that is given by the generator of the contact symmetry. 
More precisely, if we consider the generator $\Omega(t,x,u,u_{x})$ of a contact symmetry (which is a regular function defined on the jet space $J^1(\mathbb{R}^n,\mathbb{R})$, i.e.,~a map depending on a function $u$ and on its first derivatives $u_x$), a regular solution $U(t,x)$ to the HJB equation and the solution $X_t$ to the optimal control problem, then the process $O_t=\Omega(t,X_t,U(t,X_t),\nabla U(t,X_t))$, obtained by composing the generator $\Omega$ with the function $U$ and the process $X$, is a local martingale.\\
Furthermore, we generalize Noether theorem also to the case where the coefficients and the Lagrangian of the control problem are random. 
Indeed, we establish that Noether theorem holds also in the case of stochastic HJB equation, introduced in~\cite{peng_stochastic_1992} by Peng to study the optimal control problem with stochastic final condition or stochastic Lagrangian, provided that we restrict ourselves to a subset of Lie point symmetries (Theorem~\ref{theorem_noether2} and Corollary~\ref{cor:noether2}).\\
Finally, the present paper provides an application of our theory to a non-trivial interesting problem arising in mathematical finance, that is Merton's optimal portfolio problem. 
First proposed by Merton in~\cite{Merton1969lifetime}, this model finds nowadays many different applications and generalizations (see~\cite{rogers_optimal_2013} for a review of the original problem and various generalizations and~\cite{benth_merton_2003,biagini_robust_2017,fouque_portfolio_2017,Oksendal2016} for some more recent works on the subject). 
A particular form of Noether theorem for this problem can be found in~\cite{askenazy_symmetry_2003}. 
We show here that the HJB equation of this optimal control system admits infinitely many contact symmetries. 
It is important to notice that the contact symmetry generalization is essential in this specific problem, since, when we restrict to Lie point symmetries as it is done in the aforementioned literature, the equation admits only a finite number of infinitesimal invariants. 
The presence of infinitely many contact symmetries yields the possibility to construct infinitely many martingales whose means are preserved by the evolution of the system. 
Moreover, we also point out that, when the final condition is random or the coefficients of the evolution of the stock are general adapted processes, our stochastic generalization of Noether theorem (Corollary~\ref{cor:noether2}) allows us to construct some non-trivial martingales for this classical mathematical model. 
We think that the presence of these martingales could be related to the existence of many explicit solutions for Merton's problem, and therefore we expect that the methods presented here can be used to build other explicit solutions for~it. 
We plan to study in a future work the financial consequences of the conservation laws individuated in this paper.

\medskip

Since the stochastic and geometrical frameworks are not so commonly put together, we also provide a concise introduction to both these subjects. 

\subsection*{Plan of the paper}

The paper is organized as follows.
Section~\ref{s:opt-ctrl} introduces stochastic optimal control problems both in the deterministic and stochastic case, presenting also the HJB equation, and it is useful also to fix the notations that we adopt throughout the paper. 
Contact symmetries and their properties in the PDEs setting are discussed in Section~\ref{s:sym}. 
Section~\ref{s:noether} contains the main theoretical results of the paper, namely Noether theorems for deterministic and stochastic HJB equations. 
The application of such results to Merton's optimal portfolio problem is given in Section~\ref{s:merton}.

\section{A brief survey on stochastic optimal control}\label{s:opt-ctrl}

We give here an overview of some results about stochastic control problems, referring the interested reader to~\cite{fleming_deterministic_1975,pham_continuous-time_2009,touzi_optimal_2012,yong_stochastic_1999}
for further investigations on such results, though more precise references
will be given throughout the section. 
The main aim of this section is to introduce
the topics we will deal with and to give the tools from the stochastic
optimal control theory that we will use later on in the paper.

\subsection{Deterministic optimal control and Lagrange mechanics} \label{sec:lagrange}

We start recalling some notions about deterministic optimal control and, in particular, we focus on Lagrangian-type optimal control problems, i.e., problems arising from Lagrangian formulation of classical mechanics. 
More precisely, we consider a system of controlled ODEs of the form
\begin{equation}
  \mathd X_t^i = \alpha_t^i \mathd t, \label{eq:lagrange1}
\end{equation}
where $X_{\cdot} = (X^1_{\cdot}, \ldots, X_{\cdot}^n)\colon [t_0, T] \rightarrow
\mathbb{R}^n$ is in $C^1 ([t_0, T], \mathbb{R}^n)$, $t_0, T \in \mathbb{R}$,
with $t_0 \leqslant T$, are the initial time and the final time horizon,
respectively, and $\alpha_{\cdot} = (\alpha^1_{\cdot}, \ldots,
\alpha^n_{\cdot}) \in C ([t_0, T], \mathbb{R}^n)$ is the control
function. We want to maximize the following objective functional
\begin{equation}
  J (t_0, x, \alpha) = \int_{t_0}^T L (X_s^{t_0, x}, \alpha_s) \mathd s + g
  (X_T^{t_0, x}) . \label{eq:lagrange2}
\end{equation}
where $X^{t_0, x}_t$ is the solution to the ODE {\eqref{eq:lagrange1}} such
that $X^{t_0, x}_{t_0} = x \in \mathbb{R}^n$.

We suppose that there exists only one smooth function $\mathcal{A} \colon \mathbb{R}^n \times
\mathbb{R}^n \rightarrow \mathbb{R}^n$ such that
\[ \sum_{i = 1}^n \mathcal{A}^i (x, p) p_i + L (x, \mathcal{A} (x, p)) = \underset{a \in
   \mathbb{R}^n}{\sup} \left\{ \sum_{i = 1}^n a^i p_i + L (x, a) \right\}, \qquad (x,p) \in \mathbb{R}^n \times \mathbb{R}^n, \]
and also that, for any $x \in \mathbb{R}^n$, the map $\mathcal{A} (x, \cdot) \assign
(\mathcal{A}^1 (x, \cdot), \ldots, \mathcal{A}^n (x, \cdot))$ is smoothly invertible in all its
variables as a function from $\mathbb{R}^n$ into itself. Define then the
PDE
\begin{equation}
  u_t - H (x, u_x) = u_t - \left( \sum_{i = 1}^n \mathcal{A}^i (x, u_x) \, u_{x^i} + L (x,
  \mathcal{A} (x, u_x)) \right) = 0, \label{eq:lagrange3}
\end{equation}
where $u_x = (u_{x^1}, \ldots, u_{x^n})$. Equation~{\eqref{eq:lagrange3}} is
usually referred to as {\emph{Hamilton-Jacobi equation}} in the context of
Lagrangian mechanics or {\emph{Hamilton-Jacobi-Bellman equation}} in the one
of optimal control theory.

We state now the deterministic version of the so-called \emph{verification theorem}.
\begin{theorem}
  Let \ $U (t, x) \in C^1 ([t_0, T] \times \mathbb{R}^n, \mathbb{R})$ be a
  solution to Hamilton-Jacobi equation~{\eqref{eq:lagrange3}}. Then the
  optimal control problem~{\eqref{eq:lagrange1}} with objective
  functional~{\eqref{eq:lagrange2}} admits a unique solution, for any $x \in
  \mathbb{R}^n$, given, for every $i=1,\ldots,n$,~by
  \[ \alpha^i_t = \mathcal{A}^i (X_t, \nabla U (t, X_t)), \qquad \text{for every } t\in [t_0,T] . \]
\end{theorem}

\begin{proof}
  See, e.g., Theorem~4.4 in~\cite{fleming_deterministic_1975}.
\end{proof}

\begin{remark}
  \label{remark:EulerLagrange}It is important to note that, in the
  deterministic case and when $U \in C^{1, 2} ([t_0, T] \times \mathbb{R}^n,
  \mathbb{R})$, i.e.,~$U$ is differentiable one time with respect to time $t$
  and two times with respect to space $x \in \mathbb{R}^n$, the function $t
  \mapsto \alpha_t$ is $C^1 ([t_0, T], \mathbb{R}^n)$ and it satisfies the
  \tmtextit{Euler-Lagrange equations}
  \begin{equation}
    \frac{\mathd}{\mathd t} (\partial_{a^i} L (X_t, \alpha_t)) -
    \partial_{x^i}^{} L (X_t, \alpha_t) = 0, \label{eq:EulerLagrange}
  \end{equation}
  where $i = 1, \ldots, n$.
\end{remark}

\subsection{Classical stochastic optimal control problem}\label{s:det-hjb}

An optimal control problem consists in maximizing an objective
functional, depending on the state of a dynamical system, on which we can act
through a control process.

Let $K$ be a (convex) subset of~$\mathbb{R}^d$ and fix a final time $T > 0$.
Denote by $W$ an $m$-dimensional Brownian motion on a filtered probability
space $(\Omega, \mathcal{F}, (\mathcal{F}_t)_{t \geq 0}, \mathbb{P})$, where
$(\mathcal{F}_t)_{t \geq 0}$ is the natural filtration generated by~$W$. We
assume that the state of the system is modeled by the following stochastic
differential equation~(SDE)
\begin{equation}\label{eq:sde-socp}
	\begin{cases}
	 \mathd X_t = \mu (t,X_t, \alpha_t) \mathd t + \sigma (t,X_t, \alpha_t) \mathd
	 W_t, & t_0 < t \leq T,\\
	 X_{t_0} = x, & 
	\end{cases}
\end{equation}
where $\mu \colon \mathbb{R}_+ \times \mathbb{R}^n \times \mathbb{R}^d \rightarrow \mathbb{R}^n$ and
$\sigma \colon \mathbb{R}_+ \times \mathbb{R}^n \times \mathbb{R}^d \rightarrow \mathbb{R}^{n \times
m}$ are measurable functions that are also Lipschitz-continuous on the set~$K$, i.e., there exists a constant $C \geqslant 0$, such that, for every $t\in \mathbb{R}_+$, $x, y
\in \mathbb{R}^n$, $a \in K$,
\begin{equation}
  \lvert \mu (t,x, a) - \mu (t,y, a) \rvert + \lVert \sigma (t,x, a) - \sigma (t,y, a) \rVert \leqslant C
  \lvert x - y \rvert , \label{eq:lip-mu-sigma}
\end{equation}
where $\lVert \sigma \rVert^2 = \operatorname{tr} (\sigma^\ast \sigma) $.
We will also use the notation $\mu = (\mu^i)_{i = 1, \ldots, n}$ and $\sigma =
(\sigma^i_{\ell})_{i = 1, \ldots, n, \ell = 1, \ldots, m}$.

The {\emph{control}} process $\alpha = (\alpha_s)$, appearing in~\eqref{eq:sde-socp}, is a $K$-valued progressively measurable process with respect to the filtration~$(\mathcal{F}_t)_{t \geq 0}$. 
We denote by $\mathcal{K}$ the set of control processes $\alpha$ such that
\begin{equation}
  \mathbb{E} \left[ \int_0^T (\lvert \mu (t, 0, \alpha_t) \rvert^2 + \lVert \sigma (t, 0, \alpha_t)
  \rVert^2) \mathd t \right] < + \infty . \label{eq:cond-mu-sigma-ctrl}
\end{equation}
We call $X^{t_0, x}_t$, $t \in [t_0, T]$ the solution to the SDE~{\eqref{eq:sde-socp}}.

\begin{remark}
  Conditions {\eqref{eq:lip-mu-sigma}}--{\eqref{eq:cond-mu-sigma-ctrl}} imply
  that, for any initial condition $(t_0, x) \in [0, T) \times \mathbb{R}^n$
  and for all $\alpha \in \mathcal{K}$, there exists a unique strong solution
  $X^{x, t_0}_t$ to the SDE~{\eqref{eq:sde-socp}} (see, e.g., Theorem~2.2 in Chapter~4 of~\cite{ikeda_stochastic_1989}).
\end{remark}

Let now $L \colon \mathbb{R}_+ \times \mathbb{R}^n \times \mathbb{R}^d \rightarrow \mathbb{R}$ and $g \colon
\mathbb{R}^n \rightarrow \mathbb{R}$ be two measurable functions such that
\begin{enumerateroman}
  \item $g$ is bounded from below,
  \item $g$ satisfies the quadratic growth condition $| g (x) | \leqslant C (1
  + | x |^2)$, for every $x \in \mathbb{R}^n$, for some constant $C$
  independent of~$x$.
\end{enumerateroman}
For $(t_0, x) \in [0, T) \times \mathbb{R}^n$, we denote by $\mathcal{K}_L
(t_0, x)$ the subset of controls in $\mathcal{K}$ such that
\[ \mathbb{E} \left[ \int_{t_0}^T | L (t,X_t^{t_0, x}, \alpha_t) | \mathd t
   \right] < + \infty . \]
We consider an {\emph{objective function}} of the following form
\begin{equation*}
  J (t_0, x, \alpha) =\mathbb{E} \left[ \int_{t_0}^T L (s,X_s^{t_0, x},
  \alpha_s) \mathd s + g (X^{t_0, x}_T) \right] . 
\end{equation*}
We are now in position to introduce the {\emph{stochastic optimal control
problem}}.

\begin{definition}
  \label{def:socp}Fixed $(t_0,x) \in [0,T) \times \mathbb{R}^n$, the {\emph{stochastic optimal control
  problem}} consists in maximizing the objective function $J (t_0, x, \alpha)$ over
  all $\alpha \in \mathcal{K}_L (t_0, x)$ subject to the
  SDE~{\eqref{eq:sde-socp}}. The associated {\emph{value function}} is then
  defined~as
  \begin{equation*}
    U (t_0, x) = \max_{\alpha \in \mathcal{K}_L (t_0, x)} \mathbb{E} \left[
    \int_{t_0}^T L (t,X^{t_0, x}_t, \alpha_t) \mathd t + g (X^{t_0, x}_T) \right] .
  \end{equation*}
  Given an initial condition $(t_0, x) \in [0, T) \times \mathbb{R}^n$, we
  call $\alpha^{\ast} \in \mathcal{K}_L (t_0, x)$ an {\emph{optimal
  control}}~if
  \[ J (t_0, x, \alpha^{\ast}) = U (t_0, x) . \]
\end{definition}

We call {\emph{Hamilton-Jacobi-Bellman equation}} (HJB) the PDE
\begin{equation}\label{eq:hjb1}
\begin{cases} 
    \displaystyle{\partial_t \varphi (t, x) + \sup_{a \in K} \{ \mathcal{L}^a_t
    \varphi (t, x) + L (t,x, a) \} = 0,} & (t, x) \in [t_0, T) \times
    \mathbb{R}^n,\\
    \varphi (T, x) = g (x), & x \in \mathbb{R}^n,
\end{cases} 
\end{equation}
where $\mathcal{L}^a_t$ is the {\emph{Kolmogorov operator}} associated with
equation~{\eqref{eq:sde-socp}}, namely, for $\psi \in C^2 (\mathbb{R}^n)$,
\[ \mathcal{L}^a_t \psi (x) = \frac{1}{2}  \sum_{i, j = 1}^n \eta^{i j} (t,x, a)
   \partial_{x^i x^j} \psi (x) + \sum_{i = 1}^n \mu^i (t,x, a) \partial_{x^i}
   \psi (x), \qquad  (t,x, a) \in \mathbb{R}_+ \times \mathbb{R}^n \times K, \]
with $\eta^{i j}$ defined, for every $i, j \in \{ 1, \ldots, n \}$, as
\[ \eta^{i j} (t,x, a) = (\sigma \sigma^{\top})^{i j}  (t,x, a) = \sum_{\ell =
   1}^m \sigma^i_{\ell} (t,x, a) \sigma^j_{\ell} (t,x, a), \qquad (t,x, a)
   \in \mathbb{R}_+\times \mathbb{R}^n \times K. \]
We also write, for $x \in \mathbb{R}^n$, $p \in \mathbb{R}^n$ and $q \in
\mathbb{R}^{n \times n}$,
\[ H (t,x, p, q) = \sup_{a \in K} \left\{
   \frac{1}{2}  \sum_{i, j = 1}^n \eta^{i j} (t,x, a) q^{i j} + \sum_{i = 1}^n
   \mu^i (t,x, a) p^i + L (t,x, a) \right\}, \]
so that the HJB equation~{\eqref{eq:hjb1}} can be written also in the
following way
\begin{equation}
  \begin{cases}
    \partial_t \varphi (t, x) + H (t,x, \nabla \varphi, D^2 \varphi) =
    0, & (t, x) \in [t_0, T) \times \mathbb{R}^n,\\
    \varphi (T, x) = g (x), & x \in \mathbb{R}^n .
  \end{cases} \label{eq:hjbH}
\end{equation}

We state here the classical {\emph{verification theorem}}.
\begin{theorem}
  \label{thm:verification1}Let $\varphi \in C^{1, 2} ([0, T) \times
  \mathbb{R}^n) \cap C^0 ([0, T] \times \mathbb{R}^n)$ be a solution to the
  HJB equation~{\eqref{eq:hjbH}} for $t_0 = 0$, satisfying the following
  quadratic growth, for some constant $C$,
  \begin{equation}
    | \varphi (t, x) | \leqslant C (1 + | x |^2), \qquad \text{for all } (t,x) 
    \in [0, T] \times \mathbb{R}^n . \label{eq:quadratic}
  \end{equation}
  Suppose that there exists a measurable function $A^{\ast} (t, x)$, $(t, x)
  \in [0, T) \times \mathbb{R}^n$, taking values in $K$, such that
  \begin{enumerateroman}
    \item We have
    \[ 0 = \partial_t \varphi (t, x) + H (t,x, \nabla \varphi,
       D^2 \varphi) = \partial_t \varphi (t, x)
       +\mathcal{L}^{A^{\ast} (t, x)}_t \varphi (t, x) + L (t,x, A^{\ast} (t, x)),
    \]
    \item The SDE
    \[ \mathd X_s = \mu (s,X_s, A^{\ast} (s, X_s)) \mathd s + \sigma (s,X_s, A^{\ast} (s,
       X_s)) \mathd W_s, \]
    with initial condition $X_t = x$, admits a unique solution $X^{\ast}_s$,
    \item The process $A^{\ast} (s, X^{\ast}_s)$, $s \in [t, T]$ lies in
    $\mathcal{K}_L (t, x)$. 
  \end{enumerateroman}
  Then
  \[ \varphi (t,x) = U(t,x), \qquad (t,x) \in [0, T] \times \mathbb{R}^n, \]
  and $A^{\ast}(\cdot,X^\ast_\cdot)$ is an optimal control for the stochastic optimal control
  problem in Definition~\ref{def:socp}.
\end{theorem}

\begin{proof}
  See, e.g., Theorem~3.5.2 in~{\cite{pham_continuous-time_2009}}. Some other
  references for the verification theorem are also Theorem~4.1
  in~{\cite{fleming_deterministic_1975}}, Theorem~5.7
  in~{\cite{touzi_optimal_2012}}, and Theorem~4.1
  in~{\cite{yong_stochastic_1999}}.
\end{proof}

\begin{remark}
  The quadratic growth condition {\eqref{eq:quadratic}} is used in Theorem
  \ref{thm:verification} only to guarantee that the local martingale part of
  the semi-martingale decomposition of $\varphi (t, X_t)$, namely, by It{\^o}
  formula,
  \begin{equation}
  	\sum_{i=1}^n \sum_{\ell = 1}^m \int_0^t \sigma^i_{\ell} (s,X_s, A^{\ast} (s, X_s))
    \partial_{x^i} \varphi (s, X_s) \mathd W^{\ell}_s,
    \label{eq:verification1}
  \end{equation}
  is in $L^1$ and a martingale (and not only a local martingale). This means
  that the statement of Theorem~\ref{thm:verification1} holds assuming only
  that~{\eqref{eq:verification1}} is a $L^1$ martingale, i.e., without
  condition~{\eqref{eq:quadratic}}.
\end{remark}

\subsection{Stochastic Hamilton--Jacobi--Bellman equation}\label{s:stoch-hjb}

The present section generalizes the aforementioned Hamilton--Jacobi--Bellman
equation to its stochastic counterpart. Let us first recall the
{\emph{It{\^o}--Kunita formula}}.

\begin{theorem}[It{\^o}--Kunita formula]
  \label{thm:ito-kunita}Let $F (t, x)$, $(t, x) \in [0, T] \times
  \mathbb{R}^n$ be a random field which is continuous in $(t, x)$ almost
  surely, such that
  \begin{enumerateroman}
    \item For every $t \in [0, T]$, $F (t, \cdot)$ is a $C^2$-map from
    $\mathbb{R}^n$ into $\mathbb{R}$, $\mathbb{P}$-a.s.,
    \item For each $x \in \mathbb{R}^n$, $F (\cdot , x)$ is a continuous
    semi-martingale $\mathbb{P}$-a.s., and it satisfies
    \[ F (t, x) = F (0, x) + \sum_{j = 1}^m \int_0^t f^j (s, x) \mathd Y^j_s,
       \qquad \text{for every } (t,x) \in [0,T]\times \mathbb{R}^n, \text{ a.s.}, \]
    where $Y^j_s$, $j = 1, \ldots, m$, are $m$ continuous semi-martingales,
    $f^j (s, x)$, $x \in \mathbb{R}^n$, $s \in [0, T]$, are random fields that
    are continuous in $(s, x)$ and satisfy the following properties:
    \begin{enumeratealpha}
      \item For every $s \in [0, T]$, $f^j (s, \cdot)$ is a $C^2$-map from
      $\mathbb{R}^n$ to $\mathbb{R}$, $\mathbb{P}$-a.s.,
      \item For every $x \in \mathbb{R}^n$, $f^j (\cdot , x)$ is an adapted
      process.
    \end{enumeratealpha}
  \end{enumerateroman}
  Let $X_t = (X^1_t, \ldots, X_t^n)$ be continuous semi-martingales Then we
  have, for $t\in [0,T]$,
  \begin{equation*}
  \begin{split} 
    F (t, X_t) = & \, F (0, X_0) + \sum_{j = 1}^m \int_0^t f^j (s, X_s) \mathd
    Y^j_s + \sum_{i = 1}^n \int_0^t \partial_{x_i} F (s, X_s) \mathd X^i_s\\
    & + \sum_{j = 1}^m \sum_{i = 1}^n \int_0^t \partial_{x_i} f^j (s, X_s)
    \mathd [ Y^j, X^i ]_s + \sum_{i, k = 1}^n \int_0^t
    \partial_{x_i x_k} F (s, X_s) \mathd [ X^i, X^k ]_s,
  \end{split} 
  \end{equation*}
  where $[ \cdot, \cdot ]_s$ stands for the quadratic variation of
  semi-martingales. Furthermore, if $F \in C^3$ and $f^j \in C^3$,
  $\mathbb{P}$-a.s., then we have, for $i=1,\ldots,n$,
  \[ \partial_{x^i} F (t, x) = \partial_{x^i} F (0, x) + \sum_{j = 1}^m
     \int_0^t \partial_{x^i} f^j (s, x) \mathd Y^j_s, \qquad \text{for every }
     (t,x) \in [0,T] \times \mathbb{R}^n, \text{ $\mathbb{P}$-a.s.} \]
\end{theorem}

\begin{proof}
  See, e.g., the article~{\cite{kunita_some_1981}} or the
  book~{\cite{kunita_stochastic_1990}}, both by H.~Kunita.
\end{proof}

Sticking, where possible, with the notation introduced in
Section~\ref{s:det-hjb}, we consider a stochastic optimal control problem
where also the functions $L$, $g$, $\mu$ and $\sigma$ are random. 
More precisely, they depend also on $\omega \in \Omega$ in a predictable way, namely, $L(t,x,a,\cdot), g(x,\cdot), \mu(t,x,\cdot), \sigma(t,x,\cdot)$ are $\mathcal{F}_t$-measurable, for any $(t,x,a) \in \mathbb{R}_+\times \mathbb{R}^n \times K$.  
In order to distinguish them from the functions in the previous section and recall that the following are stochastic terms, we write also
\[ L^S(t,x,a) = L(t,x,a,\cdot), \qquad g^S(x) = g(x,\cdot). \]
We want then to maximize the objective functional
\begin{equation}
  \mathbb{E} \left[ \int_{t_0}^T L^S (t,X_t, \alpha_t) \mathd t + g^S (X_T) \right], \label{eq:cost-stoch}
\end{equation}
where $X$ solves the SDE
\begin{equation}
  \begin{cases}
    \mathd X_t = \mu (t,X_t, \alpha_t, \omega) \mathd t + \sigma (t,X_t, \alpha_t, \omega)
    \mathd W_t, & t_0 < t < T,\\
    X_{t_0} = x. & 
  \end{cases} \label{eq:SDE-stoch}
\end{equation}
and $\alpha \in \mathcal{K}_{L}$.

Let us introduce, in a completely analogous way as in the previous section,
the {\emph{value function}}
\[ U (t, x , \omega) = \max_{\alpha \in \mathcal{K}_{L}} \mathbb{E} \left[ 
   \int_t^T L^S (s,X_s, \alpha_s) \mathd s + g^S (X_T) \, \bigg\vert \, \mathcal{F}_t \right] . \]
From now on, we may omit the explicit dependence on $\omega\in \Omega$ of the functions.
Then, for any fixed $x$, $U (t, x)$ is an $\mathcal{F}_t$-adapted process but,
a priori, it is not of bounded variation. 
We can anyway expect that it is a
continuous semi-martingale, and therefore, by the representation theorem for
semi-martingales and martingales (see, e.g., Section~IV.31 and Section~IV.36
in~{\cite{rogers_diffusions_2000}}), that it can be written as follows,
\begin{equation*}
  U (t, x) = \Gamma_T (x) - \Gamma_t (x) - \int_t^T Y_s (x) \mathd W_s, \qquad
  x \in \mathbb{R}^n, 0 \leqslant t \leqslant T, 
\end{equation*}
where, for every $x \in \mathbb{R}^n$, $\Gamma_t (x)$ and $Y_t (x)$ are
$\mathcal{F}_t$-adapted processes and $\Gamma_t (x)$ is of bounded variations.
In this case, if $\Gamma_t (x)$ and $Y_t (x)$ are almost surely continuous in
$(t, x)$, $\Gamma_t (x)$ is differentiable with respect to $t$, and both of
them are sufficiently smooth with respect to~$x$, then the pair $(U, Y)$
should satisfy a {\emph{stochastic Hamilton--Jacobi--Bellman equation}}~(SHJB).

More precisely, we say that $(\varphi, \Psi)$ solves the SHJB related with
the optimal control problem~{\eqref{eq:cost-stoch}} and~{\eqref{eq:SDE-stoch}}
if $(\varphi, \Psi)$ satisfies the following backward stochastic partial
differential equation
\begin{equation}
  \begin{cases}
    \displaystyle{\mathd \varphi_t (x) + H^S (r,x, \nabla \varphi, D^2 \varphi, \nabla \Psi) \mathd t = \sum_{\ell = 1}^m
    \Psi^{\ell}_t (x) \mathd W^{\ell}_t,} & (t, x) \in [t_0, T) \times
    \mathbb{R}^n,\\
    \varphi_T (x) = g (x), & x \in \mathbb{R}^n,
  \end{cases} \label{eq:stoch-hjb}
\end{equation}
where
\begin{equation*}
  H^S (t,x, u_{x}, u_{x x}, \psi_{x}) \equiv  H (t,x, u_{x}, u_{x x}, \psi_{x}, \omega) = \sup_{a \in
  K} \mathcal{H}^S (t,x, u_{x}, u_{x x}, a, \psi_{x}),
\end{equation*}
and
\begin{equation*} 
\begin{split}
  \mathcal{H}^S (t,x, u_{x}, u_{x x}, a, \psi_{x}) = & \, \sum_{i=1}^n \mu^i (t,x, a,\omega) u_{x^i} + \frac{1}{2} \sum_{i,j=1}^n\sum_{\ell = 1}^m \sigma^i_{\ell} (t,x, a,
  \omega) \sigma^j_{\ell} (t,x, a, \omega) u_{x^i x^j} \\
  & + \sum_{i=1}^n \sum_{\ell = 1}^m \sigma_{\ell}^i (t,x, a, \omega)
  \psi^{\ell}_{x^i} + L^S (t,x,a), 
\end{split}
\end{equation*}
with $\psi_x = (\psi^\ell_{x^i})_{i=1,\ldots,n,\ell=1,\ldots,m}\in \mathbb{R}^{n\times m}$. 
See, e.g., Section~3.1 in~{\cite{peng_stochastic_1992}} for more details about
the derivation of equation~{\eqref{eq:stoch-hjb}} and Section~4 in the same
reference for results concerning the well-posedness of such an equation.

We state here the {\emph{verification theorem}}, which tells us that a sufficiently smooth
solution of the SHJB equation coincides with the value function~$v$.

\begin{theorem}
  \label{thm:verification}Let $(\varphi, \Psi)$ be a smooth solution of the
  SHJB equation~{\eqref{eq:stoch-hjb}} with $t_0 = 0$ and assume that the following
  conditions hold:
  \begin{enumerateroman}
    \item For each $t \in [0,T]$, $x\mapsto (\varphi_t (x), \Psi_t (x))$ is a $C^2$-map from
    $\mathbb{R}^n$ into $\mathbb{R} \times \mathbb{R}^m$, $\mathbb{P}$-a.s.,
    
    \item For each $x \in \mathbb{R}^n$, $t \mapsto (\varphi_t (x), \Psi_t (x))$ and $t \mapsto (\nabla \varphi_t (x),
    D^2 \varphi_t (x), \nabla \Psi_t (x))$ are continuous $\mathcal{F}_t$-adapted
    processes.
  \end{enumerateroman}
  Suppose further that there exists a predictable admissible control $A^\ast (t,x,\omega)$ such that
  \[ H^S (t,x, \nabla\varphi, D^2\varphi, \nabla \Psi) = \mathcal{H}^S
     (t,x, \nabla \varphi, D^2 \varphi, A^\ast(t,x,\omega), \nabla \Psi), \]
  and that it is regular enough so that the SDE~{\eqref{eq:SDE-stoch}} is
  well-posed with solution~$X$. Then $(\varphi,
  \Psi) = (V, Y)$ and moreover, for any initial data $(0, x)$ with $x \in \mathbb{R}^n$, $A^{\ast}  (t,X_t, \omega)$ maximizes the objective function~$U$.
\end{theorem}

\begin{proof}
  See Section~3.2 in~{\cite{peng_stochastic_1992}}.
\end{proof}

\begin{remark}
  Under suitable regularity conditions on $\mu$, $\sigma$, $L^S$ and $g^S$, it is
  possible to prove that the SHJB equation~{\eqref{eq:stoch-hjb}} admits a
  unique solution satisfying the hypotheses of Theorem~\ref{thm:verification}.
  A rigorous proof of this fact can be found in~Section~4
  of~{\cite{peng_stochastic_1992}}. For further developments on SHJB equations
  and the related stochastic optimal control problems we refer the reader to,
  e.g.,~{\cite{buckdahn_pathwise_2007,chang_optimal_2009,englezos_utility_2009,qiu_viscosity_2018}},
  as well as the already mentioned paper by
  Peng~{\cite{peng_stochastic_1992}}.
\end{remark}

\section{Solutions of PDEs via contact symmetries}\label{s:sym}

In this section, we recall some basic facts concerning the theory of symmetries on which our results are based, referring to~{\cite{devecchi_geometry_2020,gaeta_symmetries_1994,hydon_symmetry_2000,olver_applications_1993,stephani_differential_1989}} for a complete treatment of these topics. 
We start with a formal introduction on jet spaces (for an extended introduction to the subject see, e.g., \cite{bocharov_symmetries_1999,Saundersbook}), and then proceed with contact symmetries and their applications in solving~PDEs. 
Despite the fact that these results are well-known, we insert here a small survey for the ease of the reader, as well as we introduce the notation that will be adopted in the rest of the paper.

\subsection{Jet spaces and jet bundles}

The jet space is a generalization of the notion of tangent bundle of a
manifold. 
Let $M$ and $N$ be two open subsets of $\mathbb{R}^m$ and
$\mathbb{R}^n$, respectively, and consider a smooth function $f \colon M
\rightarrow N$. 
Take a standard coordinate system $x = (x^1, \ldots, x^m)$ in
$M$ and let $u = (u^1, \ldots, u^n) = f (x) \in N$. 
We can then consider the {\emph{$k$-th prolongation}} $u^{(k)} = \tmop{pr}^{(k)} f (x)$, that is defined by the relations $u^j_{x^i} = \partial_{x^i} f^j (x), u^j_{x^i x^l} = \partial_{x^i x^l} f^j (x), \ldots$, up to order~$k$.
For example, if $m = 2$ and $n = 1$, then $\tmop{pr}^{(2)} f (x^1,
x^2)$ is given~by
\[ (u ; u_{x^1}, u_{x^2} ; u_{x^1 x^1}, u_{x^1 x^2}, u_{x^2 x^2}) = (f ; \partial_{x^1} f, \partial_{x^2} f ; \partial_{x^1 x^1} f, \partial_{x^1 x^2} f, \partial_{x^2 x^2} f) (x^1, x^2) . \]
The $k$-th prolongation can also be looked at as the Taylor polynomial of
degree $k$ for $f$ at the point~$x$. 
The space whose coordinates represent the
independent variables, the dependent variables and the derivatives of the
dependent variables up to order $k$ is called the {\emph{$k$-th order jet
space}} of the underlying space $N \times M$, and we denote it by $J^k (M,
N)$. It is a smooth vector bundle on $M$ with projection $\pi_{k, - 1} \colon J^k
(M, N) \rightarrow M$ given~by
\[ \pi_{k, - 1} (x, u, u_{x}, u_{x x}, \ldots) = x. \]
More explicitly, $J^k (M, N) = M \times N \times N_{(1)} \times \cdots \times
N_{(k)}$, where $N_{(i)}$, is the space of $i$-th order derivatives of $u$
with respect to~$x$. 
It is clear that $N_{(i)} \subseteq \mathbb{R}^{n_i}$
with
\[ n_i = \left( \begin{array}{c}
     m + i - 1\\
     i
   \end{array} \right) . \]

To any function $f \in C^k (M, N)$, where $C^k (M, N)$ is the
infinite-dimensional Fr{\'e}chet space of $k$ times differentiable functions
on $M$ taking values in $N$, we associate a continuous section of the bundle
$(J^k (M, N), M, \pi_{k, - 1})$ in the following way
\[ f \mapsto \boldsymbol{D}^k (f) (x) = (x, u = f (x), u_{x} =
   \nabla f (x), u_{x x} = D^2 f(x), \ldots, D^k f(x)), \]
where $D^i f(x)$ is the vector collecting all the $i$-th order derivatives of $f$ with respect to~$x$.
In this setting, a differential equation is a sub-manifold $\Delta_\mathcal{E} \subset
J^k (M, N)$. For example, in the scalar case $N =\mathbb{R}$, we usually
consider $\Delta_\mathcal{E}$ as the null set of some regular functions, i.e.,~$\Delta_\mathcal{E} =
\{ E^i (x, u, u_{x},u_{xx},\ldots) = 0, i \in \{ 1, \ldots, p \} \}$.

\begin{definition}
 Consider a (finite) set $E^i\colon J^k(M,N) \rightarrow \mathbb{R}$, for $i=1,...,p$ where $p\in \mathbb{N}$ and $p>0$, of smooth functions defining a sub-manifold $\Delta_\mathcal{E} =\{ E^i (x, u, u_{x},u_{xx},\ldots) = 0, i \in \{ 1, \ldots, p \} \}$ of $J^k(M,N)$. We say that a smooth function $f \colon M \rightarrow N$ is a {\emph{solution to
  the equation}} $\mathcal{E}$ (represented by the sub-manifold $\Delta_{\mathcal{E}}$) if, for any $x \in M$, we have $\boldsymbol{D}^k
  f (x) \in \Delta_{\mathcal{E}}$. The set of all solutions to equation
  $\mathcal{E}$ will be denoted by~$\mathcal{S}_{\mathcal{E}}$.
\end{definition}

For instance, in the previous case where $N =\mathbb{R}$ and $\Delta_\mathcal{E} = \{
E^i (x, u, u_{x}, \ldots) = 0, i \in \{ 1, \ldots, p \} \}$, $f$ is a solution to
equation $\mathcal{E}$ if $E^i (x, f (x), \nabla f (x), \ldots)
= 0$, for every $i = 1, \ldots, p$, $x \in M$.

\begin{remark}\label{remark:nondegenerate}
For technical reasons, it is usually not possible to consider generic equations $\mathcal{E}$ (corresponding to generic sub-manifold $\Delta_{\mathcal{E}} \subset J^k(M,N)$). 
In the following, we always consider \emph{non-degenerate systems of differential equations} in the sense of Definition~2.70 in~\cite{olver_applications_1993}. 
This condition assures that, for each fixed $x_0 \in M$ and each set of derivatives $(u^0,u^0_x,u^0_{xx},\ldots)$, there exists a solution to the equation defined in a neighborhood of $x_0$ with the prescribed derivatives $(u^0,u^0_x,u^0_{xx},\ldots)$ at the point $x_0$. 
Since the precise formulation of this condition is quite technical and the evolution equations considered in Section~\ref{s:noether} always satisfy such an assumption, we refer to Section~2.6 of~\cite{olver_applications_1993} for complete details.
\end{remark}

\subsection{Contact transformations}

We want to introduce a class of transformations induced by diffeomorphisms of $J^k(M,N)$. 
For simplicity, we consider the case $k=2$, $M\subset \mathbb{R}^n$ and $N=\mathbb{R}$. 
Consider a diffeomorphism $\Phi \colon J^2(M,N)\rightarrow J^2(M,N)$ given by the following relations
\begin{align*}
\tilde{x}^i&=\Phi^{x^i}(x,u,u_{x},u_{xx}), \\
\tilde{u}&=\Phi^{u}(x,u,u_{x},u_{xx}), \\
\tilde{u}_{x^i}&=\Phi^{u_{x^i}}(x,u,u_{x},u_{xx}), \\
\tilde{u}_{x^ix^j}&=\Phi
^{u_{x^ix^j}}(x,u,u_{x},u_{xx}).
\end{align*}
Hereafter, we use the notation $\Phi^x=(\Phi^{x^1},\cdots,\Phi^{x^n})$, $\Phi^{u_x}=(\Phi^{u_{x^1}},\cdots,\Phi^{u_{x^n}})$ and $\Phi^{u_{xx}}=(\Phi^{u_{x^ix^j}})|_{i,j=1,...,n}$.

We now aim to define a transformation $F_{\Phi}$ on the space of smooth functions induced by the map $\Phi$ on the jet space. 
Let $U\in C^{\infty}(M,N)$ and consider the map $C_{U,\Phi} \colon M \rightarrow M$ given~by
\[ C_{U,\Phi}(x)=\Phi^x(x,U(x),\nabla U(x), D^2 U(x)). \]
Let also $\mathcal{F}_{\Phi} \subset C^{\infty} (M, N)$ be the subset of
smooth functions $U \in C^{\infty} (M, N)$ such that $C_{U,
\Phi}$ is a diffeomorphism from $M$ into itself.

\begin{definition}\label{definition:FPhi}
We say that the diffeomorphism $\Phi$ generates the (nonlinear) operator $F_{\Phi}$ on the space of functions $\mathcal{F}_{\Phi}$, if there is a map $F_{\Phi} \colon \mathcal{F}_{\Phi} \rightarrow C^{\infty}(M,N)$ such that 
\begin{align*}
F_{\Phi}(U)(x) & = \Phi^u(C_{U, \Phi}^{- 1} (x), U (C_{U, \Phi}^{- 1}
     (x)), \nabla U (C_{U, \Phi}^{- 1} (x)),D^2 U (C_{U, \Phi}^{- 1} (x)) ),\\
\partial_{x^i}F_{\Phi}(U)(x) & = \Phi^{u_{x^i}}(C_{U, \Phi}^{- 1} (x), U (C_{U, \Phi}^{- 1}
     (x)), \nabla U (C_{U, \Phi}^{- 1} (x)),D^2 U (C_{U, \Phi}^{- 1} (x)) ),\\
\partial_{x^ix^j}F_{\Phi}(U)(x) & = \Phi^{u_{x^ix^j}}(C_{U, \Phi}^{- 1} (x), U (C_{U, \Phi}^{- 1}
     (x)), \nabla U (C_{U, \Phi}^{- 1} (x)),D^2 U (C_{U, \Phi}^{- 1} (x)) ).
\end{align*}
\end{definition}

Not every diffeomorphism $\Phi \colon J^2(M,N)\rightarrow J^2(M,N)$ generates an operator $F_{\Phi}$ on the space of functions~$\mathcal{F}_{\Phi}$. 
For example, consider $M=\mathbb{R}$ and the map $\Phi^x(x,u,u_x,u_{xx})=\lambda x$, $\Phi^u(x,u,u_x,u_{xx})=u$, $\Phi^{u_x}(x,u,u_x,u_{xx})=u_{x}$, and $\Phi^{u_{xx}}(x,u,u_x,u_{xx})=u_{xx}$, where $\lambda>0$. 
In this case, for any $U\in C^{\infty}(M,N)$, the map $C_{U,\Phi}$ is given by $C_{U,\Phi}(x)=\lambda x$ and, thus, it does not depend on $U$ and it is always a diffeomorphism from $\mathbb{R}$ into itself, since $\lambda \ne 0$. 
This implies that $\mathcal{F}_{\Phi}=C^{\infty}(M,N)$ and also that, if the map $F_{\Phi}$ exists, then it must satisfy 
\[ F_{\Phi}(U)=U(\lambda^{-1} x), \]
for any $U\in C^{\infty}(M,N)$. On the other hand, we have 
\[ \partial_{x}F_{\Phi}(U)=\lambda^{-1} U'(\lambda^{-1} x) \not= U'(\lambda^{-1} x) = \Phi^{u_{x}}(C_{U, \Phi}^{- 1} (x), U (C_{U, \Phi}^{- 1}
     (x)), \nabla U (C_{U, \Phi}^{- 1} (x)),D^2 U (C_{U, \Phi}^{- 1} (x)) ). \]
This simple counterexample shows that a diffeomorphism $\Phi \colon J^2(M,N)\rightarrow J^2(M,N)$ must satisfy some additional conditions in order to generate an operator $F_{\Phi}$. 
For this reason, we introduce the following definition.

\begin{definition}\label{definition:contact}
	A diffeomorphism $\Phi \colon J^2(M,N)\rightarrow J^2(M,N)$ is said to be a \emph{contact transformation} if it generates a (nonlinear) operator $F_{\Phi}$ in the sense of Definition~\ref{definition:FPhi}.
\end{definition}

It is possible to give a nice geometric characterization of the set of contact transformations. 
From now on, we write $\Lambda^1 J^n(M,N)$ for the vector space of $1$-forms on $J^n(M,N)$. In particular, consider the following $1$-forms,
\begin{align}
\kappa=&\mathd u-\sum_{i=1}^n u_{x^i}\mathd x^i,\label{eq:kappa}\\
\kappa_{x^i}=&\mathd u_{x^i}-\sum_{j=1}^n u_{x^ix^j}\mathd x^j. \nonumber
\end{align}
We denote by $\mathfrak{C} \subset \Lambda^1J^2(M,N)$ the \emph{contact structure}, also called Cartan distribution in~\cite{bocharov_symmetries_1999}, which is generated~by
\[\mathfrak{C}=\operatorname{span}\{\kappa,\kappa_{x^i}, i=1,...,n\}.\] 

\begin{theorem}  \label{thm:loc-diffeo-contact}
	A diffeomorphism $\Phi \colon J^2(M,N) \rightarrow J^2(M,N)$ is a contact transformation in the sense of Definition~\ref{definition:contact} if and only if it preserves the contact structure $\mathfrak{C}$, that~is,
	\[\Phi^*(\mathfrak{C})=\mathfrak{C},\]
	where $\Phi^*$ is the pull-back of differential forms on $J^2(M,N)$ induced by~$\Phi$.
\end{theorem}

\begin{proof}
  See, e.g., Chapter~2 in~\cite{bocharov_symmetries_1999}, Section~4 in~\cite{devecchi_geometry_2020}, Chapter~21 in~\cite{stephani_differential_1989}, and the references therein.
\end{proof}

\begin{remark}\label{remark:contacthistory1}
  The contact transformation $\Phi$ is uniquely determined by its action on
  $J^1 (M, N)$. 
  In particular, $\Phi^{x},\Phi^{u}$ and $\Phi^{u_x}$ depend only on $(x,u,u_{x})$ and they do not depend on $u_{xx}$ (see, e.g., Chapter~2 in~\cite{bocharov_symmetries_1999}). 
\end{remark}
  
\begin{remark}
	In contact geometry, a contact structure on a $(2n+1)$-dimensional manifold $\mathcal{M}$ is a $1$-form $\zeta$ which is maximally non-integrable, namely, $\zeta \wedge (\mathd \zeta)^n \not=0$, and the contact transformations are the diffeomorphisms $\Psi$ of $\mathcal{M}$ such that $\Psi^*(\zeta)=f \cdot \zeta$, for some $f\in C^{\infty}(\mathcal{M},\mathbb{R})$ (see, e.g.,~\cite{arnold_geometrical_1988,Giegesbook} for an introduction to the subject and~\cite{Geiges2001} for an historical overview). 
	This definition is satisfied by $J^1(M,\mathbb{R})$ with the $1$-form $\zeta=\kappa$ defined in~\eqref{eq:kappa}.\\
	In the study of the geometry of jet spaces (see, e.g.,~Chapter~6~of~\cite{Saundersbook}), the term ``contact structure'' is often used to express the set of forms~$\mathfrak{C}$. 
	This custom is due to the fact that, as explained in Remark~\ref{remark:contacthistory1}, the contact transformations are extensions of diffeomorphisms on $J^1(M,N)$, i.e., the set of transformations considered here is in one-to-one correspondence with the one usually considered in contact geometry.
\end{remark}

In the following, we will not consider just a single contact transformation but \emph{one parameter groups of contact transformations} $\Phi_{\lambda}$, which means that $\Phi_{\cdot} \colon \mathbb{R} \times J^2(M,N) \rightarrow J^2(M,N)$ is $C^{\infty}$, $\Phi_{\lambda}$ is a contact transformation for each $\lambda \in \mathbb{R}$, $\Phi_{0}(x,u,u_{x},u_{xx})=(x,u,u_{x},u_{xx})$, and, for each $\lambda_1,\lambda_2\in \mathbb{R}$,
\[\Phi_{\lambda_1} \circ \Phi_{\lambda_2}=\Phi_{\lambda_1+\lambda_2}.\]
In general, a one parameter group of diffeomorphisms $\Phi_{Y,\lambda}$, where $\lambda\in \mathbb{R}$, is generated by a vector field $Y \in TJ^2(M,N)$, i.e., belonging to the tangent bundle of $J^2(M,N)$, which in local coordinates has the expression
\begin{equation}
	\begin{split} 
	Y = \, & \sum_{i=1}^n Y^{x^i} (x, u, u_{x},u_{xx}) \partial_{x^i} + Y^u (x, u, u_{x},u_{xx})\partial_u \\ 
	& + \sum_{i=1}^n Y^{u_{x^i}} (x, u, u_{x},u_{xx}) \partial_{u_{x^i}}
	+ \sum_{i,j=1}^nY^{u_{x^ix^j}} (x, u, u_{x},u_{xx})\partial_{u_{x^ix^j}}, 
	\end{split} 
	\label{eq:Y1}
\end{equation} 
by the following relations
\begin{align}
\nonumber \partial_{\lambda}\Phi^{x^i}_{Y,\lambda}(x,u,u_{x},u_{xx})& = Y^{x^i}\circ \Phi_{\lambda}(x,u,u_{x},u_{xx}), \\
\nonumber \partial_{\lambda}\Phi^{u}_{Y,\lambda}(x,u,u_{x},u_{xx})& =Y^{u}\circ \Phi_{\lambda}(x,u,u_{x},u_{xx}), \\
\nonumber\partial_{\lambda}\Phi^{u_{x^i}}_{Y,\lambda}(x,u,u_{x},u_{xx})& =Y^{u_{x^i}}\circ \Phi_{\lambda}(x,u,u_{x},u_{xx}), \\
\partial_{\lambda}\Phi^{u_{x^ix^j}}_{Y,\lambda}(x,u,u_{x},u_{xx})&=Y^{u_{x^ix^j}}\circ \Phi_{\lambda}(x,u,u_{x},u_{xx}),\label{eq:Y2}
\end{align}
for any $\lambda\in\mathbb{R}$ and $(x,u,u_{x},u_{xx})\in J^2(M,N)$.
It is useful to introduce the following natural notion.

\begin{definition}\label{definition:infcontact}
A vector field $Y$ (of the form~\eqref{eq:Y1}) on $J^2(M,N)$ is called an \emph{infinitesimal contact transformation} if it generates (through equation~\eqref{eq:Y2}) a one parameter group of diffeomorphisms $\Phi_{\lambda}$ of contact transformations.
\end{definition}

The following theorem characterizes all the infinitesimal contact transformations on~$J^2(M,N)$.

\begin{theorem}
  \label{thm:contact1}A vector field $Y$ on $J^2(M,N)$ is an infinitesimal contact transformation (in the sense of Definition \ref{definition:infcontact}) if and only if there exists a unique
  smooth map $\Omega \colon J^1 (M, N) \rightarrow \mathbb{R}$ such that $Y =
  Y_{\Omega}$, where $Y_{\Omega}$ is a vector field on $J^2 (M, N)$ defined~as
  \begin{equation}
  \begin{split}
    Y_{\Omega} = & \, - \sum_{i=1}^n \partial_{u_{x^i}} \Omega \partial_{x^i} +
    \left( \Omega - \sum_{i=1}^n u_{x^i} \partial_{u_{x^i}} \Omega \right)
    \partial_u + \sum_{i=1}^n (\partial_{x^i} \Omega + u_{x^i} \partial_u \Omega)
    \partial_{u_{x^i}}\\
    & + \sum_{i,j,k,\ell =1}^n \biggl( \partial_{x^i x^j} \Omega + u_{x^j}
    \partial_{x^i u} \Omega + u_{x^j x^k} \partial_{x^i u_{x^k}} \Omega +
    u_{x^i} \partial_{x^j u} \Omega + u_{x^i} u_{x^j}
    \partial_{u u} \Omega + u_{x^i} u_{x^j x^k} \partial_{u u_{x^k}} \Omega
    \\
    & + u_{x^i x^j} \partial_u \Omega + u_{x^i x^k} \partial_{x^j u_{x^k}} \Omega + u_{x^i x^k} u_{x^j} \partial_{u_{x^k} u} \Omega
    + u_{x^i x^k} u_{x^j x^{\ell}} \partial_{u_{x^k} u_{x^{\ell}}}
    \Omega \biggr) \partial_{u_{x^i x^j}} \label{eq:contact1} . 
  \end{split}
  \end{equation}
\end{theorem}

\begin{proof}
  The proof can be found in Chapter~21 of~{\cite{stephani_differential_1989}}
  and references therein.
\end{proof}

\begin{remark}
  \label{rmk:inf-cont-transf}We say that a vector field of the form
  $Y_{\Omega}$ satisfying the hypotheses of Theorem~\ref{thm:contact1} is the infinitesimal contact transformation generated by the (contact
  generating) function~$\Omega$. 
  Under this terminology,
  Theorem~\ref{thm:loc-diffeo-contact} guarantees that any infinitesimal
  contact transformation is generated in a unique way by some smooth function
  $\Omega \colon J^1 (M, N) \rightarrow \mathbb{R}$.
\end{remark}

There is a special subset of vector fields of the type~{\eqref{eq:contact1}}
arising from coordinate transformations involving only the dependent and
independent variables $(x,u)$. 

\begin{definition}
We say that $Y_{\Omega_{\mathrm{Lie},f,g}}$ is a \emph{(projected) Lie point transformation} if it is a contact transformation of the form 
\begin{equation}\label{eq:Liepoint1}
Y_{\Omega_{\mathrm{Lie},f,g}}= \sum_{i=1}^n f^i (x) \partial_{x^i} + g (x, u)\partial_{u}+\sum_{i=1}^n Y^{u_{x^i}} (x, u, u_{x})\partial_{u_{x^i}}+\sum_{i,j=1}^nY^{u_{x^ix^j}} (x, u, u_{x},u_{xx}) \partial_{u_{x^ix^j}},
\end{equation}
where $f^i\in C^{\infty}(M,\mathbb{R})$, $g\in C^{\infty}(J^{0}(M,N),\mathbb{R})$, $Y^{u_{x^i}} \in C^{\infty}(J^{1}(M,N),\mathbb{R})$ and $Y^{u_{x^ix^j}}\in C^{\infty}(J^{2}(M,N),\mathbb{R})$.
\end{definition}

\begin{remark}\label{remark:Liepoint1}
It is simple to see that a Lie point transformation $Y_{\Omega_{\text{Lie},f,g}}$ can be reduced to a standard vector
field $\tilde{Y} = \sum_i f^i (x) \partial_{x^i} + g (x, u) \partial_u$ on $J^0
(\mathbb{R}^n, \mathbb{R})$, i.e., $\tilde{Y}$ is the generator of a one parameter
group of transformations involving only the dependent and independent
variables $(x,u)$.
\end{remark}

\begin{remark}
 Another important property of Lie point transformations is the following. Denoting by $\Phi_{\text{Lie},f,g,\lambda}$, where $\lambda \in \mathbb{R}$, the one parameter group generated by the Lie point transformation $Y_{\Omega_{\text{Lie},f,g}}$, we have that, for any $\lambda \in \mathbb{R}$, the domain $\mathcal{F}_{\Phi_{\text{Lie},f,g,\lambda}}$ of the nonlinear operator $F_{\Phi_{\text{Lie},f,g,\lambda}}$ generated by $\Phi_{\text{Lie},f,g,\lambda}$ is the whole $C^{\infty}(M,N)=\mathcal{F}_{\Phi_{\text{Lie},f,g,\lambda}}$.
\end{remark}

For what follows, we introduce the (formal) operators $\mathcal{D}_{x^i}\colon C^{\infty}(J^k(M,N)) \rightarrow C^{\infty}(J^{k+1}(M,N))$ given by
\begin{equation}\label{eq:mathcalD}
\mathcal{D}_{x^i}=\partial_{x^i}+u_{x^i}\partial_{u}+\sum_{j=1}^nu_{x^ix^j}\partial_{u_{x^j}}+\ldots+\sum_{j_1\geq \ldots \geq j_p=1}^n u_{ x^{j_1} \cdots x^{j_{\ell}} x^i } \partial_{ u_{x^{j_1}\cdots x^{j_{\ell}}}}+\ldots
\end{equation}
In a similar way, we write $\mathcal{D}_{x^i x^j}(\cdot)=\mathcal{D}_{x^i}(\mathcal{D}_{x^j}(\cdot))$, $\mathcal{D}_{x^i x^jx^{\ell}}(\cdot)=\mathcal{D}_{x^i}(\mathcal{D}_{x^j}(\mathcal{D}_{x^{\ell}}(\cdot)))$, etc.\\

We can characterize more precisely the general form of Lie point transformations.

\begin{theorem}\label{theorem_Liepoint}
The vector field $Y_{\Omega_{\mathrm{Lie},f,g}}$ is a (projected) Lie point transformation if and only if it is generated by a function of the form 
\begin{equation}\label{eq:Liegenerator} 
\Omega_{\mathrm{Lie},f,g} (x, u, u_{x}) = g (x, u) - \sum_{i=1}^n f^i (x) u_{x^i}, 
\end{equation}
namely, $Y_{\Omega_{\mathrm{Lie},f,g}}$ has the following expression
\begin{equation*} 
\begin{split}
  Y_{\Omega_{\mathrm{Lie,g,f}}} \assign & \, \sum_{i=1}^n f^i (x) \partial_{x^i} + g (x, u)
  \partial_u+ \sum_{i,j=1}^n (- \mathcal{D}_{x^i} f^j (x) u_{x^j} +  \mathcal{D}_{x^i} (g))
  \partial_{u_{x^i}} \\
  & + \sum_{i,j,k=1}^n (-  \mathcal{D}_{x^i x^j} (f^k) u_{x^k} -  \mathcal{D}_{x^i} (f^k) u_{x^k
  x^j} +  \mathcal{D}_{x^i x^j} (g)) \partial_{u_{x^i x^j}} .  
\end{split}
\end{equation*} 
\end{theorem}
\begin{proof}
The theorem is a direct application of Theorem~\ref{thm:contact1} to vector fields of the form~\eqref{eq:Liepoint1}.
\end{proof}

If $n = 1$, and the coordinate system of $J^0 (\mathbb{R}, \mathbb{R})$ is
given by $(x, u)$, some examples of Lie point transformations~are:
\begin{itemize}
  \item The dilation of independent variable $x$, i.e., $\tilde{Y} = x \partial_x$ (see the notation in Remark \ref{remark:Liepoint1}),  
  related to the generator function $\Omega = - x u_x$ and generating the one
  parameter group defined~by
  \[ \Phi^x_{\lambda}(x, u) = e^{\lambda} x \qquad \Phi^{u}_{\lambda} (x, u) = u \qquad \Phi^{u_x}_{\lambda}
     (x, u, u_x) = e^{- \lambda} u_x \qquad \Phi^{u_{x x}}_{\lambda} (x, u, u_x, u_{x x}) =
     e^{- 2 \lambda} u_{x x}. \]
  \item The dilation of dependent variable $u$, namely, $\tilde{Y} = u \partial_u$
  related to the generator function $\Omega = u$ and generating the one
  parameter group defined~by
  \[ \Phi^x_{\lambda} (x, u) = x \qquad \Phi^u_{\lambda} (x, u) = e^{\lambda} u \qquad \Phi^{u_x}_{\lambda}
     (x, u, u_x) = e^{\lambda} u_x \qquad \Phi^{u_{x x}}_{\lambda} (x, u, u_x, u_{x x}) = e^{\lambda}
     u_{x x} . \]
\end{itemize}

We conclude this section providing the definition of symmetry of a differential equation.

\begin{definition}\label{definition_symmetry}
A contact transformation $\Phi \colon J^2(M,N) \rightarrow J^{2}(M,N)$ is a \emph{(contact) symmetry of  the differential equation $\mathcal{E}$} if, for any solution $U \in C^{\infty}(M,N) \cap \mathcal{F}_{\Phi}$ to the equation~$\mathcal{E}$, also $F_{\Phi}(U)$ is a solution to~$\mathcal{E}$, where $F_{\Phi}$ and $\mathcal{F}_{\Phi}$ are the operator generated by the contact transformation $\Phi$ and the domain of $F_{\Phi}$, respectively (see Definition~\ref{definition:FPhi}). \\
We say that an (infinitesimal) contact transformation $Y_{\Omega}$ is an \emph{(infinitesimal contact) symmetry of the differential equation~$\mathcal{E}$} if the one parameter group $\Phi_{Y_{\Omega},\lambda}$ generated by $Y_{\Omega}$ is a set of symmetries of the equation~$\mathcal{E}$.  
\end{definition}

\begin{remark}
With an abuse of language, we say that the function $\Omega \in C^{\infty}(J^1(M,N),\mathbb{R})$ is a \emph{contact symmetry of the equation~$\mathcal{E}$} if the corresponding contact vector field $Y_{\Omega}$ is a symmetry of~$\mathcal{E}$.
\end{remark}

\begin{remark}
If $Y$ is a Lie point transformation and it is a contact symmetry of the equation $\mathcal{E}$, then we say that $Y$ is a \emph{Lie point symmetry} of the equation~$\mathcal{E}$.
\end{remark}

It is possible to give a completely geometric characterization of the contact symmetries of a differential equation~$\mathcal{E}$.

\begin{theorem}[Determining equations]\label{theorem:determining}
A contact transformation $\Phi$ is a symmetry of the equation~$\mathcal{E}$ represented by the sub-manifold $\Delta_{\mathcal{E}} \subset J^2(M,N)$ of the form  
\[ \Delta_\mathcal{E} = \{
E^i (x, u, u_{x}, u_{xx}) = 0, i \in \{ 1, \ldots, p \} \}, \] 
where $p\in \mathbb{N}$, $p>0$ and $E^i\in C^{\infty}(J^2(M,N),\mathbb{R})$, 
if and only~if 
\[\Phi(\Delta_{\mathcal{E}})=\Delta_{\mathcal{E}}.\]
The infinitesimal contact transformation $Y_{\Omega}$ is a symmetry of the non-degenerate differential equation $\mathcal{E}$ (see Remark~\ref{remark:nondegenerate} for the definition of non-degenerate differential equation) if and only~if 
  \begin{equation}
    Y (E^{i} (x, u, u_{x},u_{xx})) \vert_{\Delta_{\mathcal{E} }} = 0, \label{eq:symmetry}
  \end{equation}
   where $i = 1, \ldots, p$.
\end{theorem}
\begin{proof}
The proof is given in Theorem~2.27 and Theorem~2.71 in~\cite{olver_applications_1993} for the case of Lie point symmetries which are diffeomorphisms of $J^k(M,N)$, for $k\geq 0$. 
Since the contact transformations are diffeomorphism of $J^h(M,N)$, for any $h\geq 1$ (see, e.g., Chapter~21 of~{\cite{stephani_differential_1989}}), the case of contact transformations can be proved using the same methods.
\end{proof}

\subsection{Symmetries and classical Noether theorem}

Let us discuss here the classical Noether theorem in the Lagrangian mechanics
setting described in Section~\ref{sec:lagrange}. 
Heuristically, Noether theorem says that to any infinitesimal transformation leaving invariant the optimal control problem, namely equation~{\eqref{eq:lagrange1}} and the Lagrangian~$L$, a constant of motion is associated.

More precisely, let $Y^{x, a}$ be a vector field in $\mathbb{R}^n \times
\mathbb{R}^n$ transforming the variables $x^i$ and $a^i$ of
equations~{\eqref{eq:lagrange1}} and the Lagrangian~$L$. 
We suppose that
$Y^{x, a}$ is ``projected'' with respect to the variables $x^i$, that~is,
\begin{equation}
  Y^{x, a} = \sum_{i = 1}^n \left( f^i (x) \partial_{x^i} + g^i (x,a) \partial_{a^i} \right). \label{eq:projvf}
\end{equation}
If we want the projected vector field~{\eqref{eq:projvf}} to be a symmetry of
equation~{\eqref{eq:lagrange1}}, then we need that
\begin{equation}
  g^i (x, a) = \sum_{j = 1}^n \partial_{x^j} f^i (x) a^j .
  \label{eq:lagrange5}
\end{equation}
If we also require that $L$ is invariant with respect to the flow of $Y^{x,
a}$, then we must have
\begin{equation}
  Y^{x, a} (L) (x, a) = \sum_{i = 1}^n f^i (x) \partial_{x^i} L (x, a) +
  \sum_{i, j = 1}^n \partial_{x^j} f^i (x) a^j \partial_{a^i} L (x, a) = 0 \label{eq:lagrange4}
  .
\end{equation}
So we say that $Y^{x, a}$ is a symmetry of the optimal control problem of
Section~\ref{sec:lagrange} if and only if conditions~{\eqref{eq:lagrange5}}
and~{\eqref{eq:lagrange4}} hold.

\begin{theorem}[Noether theorem]
  \label{theorem:cNoether1}Let $Y^{x, a}$ be a symmetry of the Lagrangian $L$
  according with equation~{\eqref{eq:lagrange4}}. 
  Then, supposing the existence of a $C^1$ optimal control $\alpha_t$, we have that
  \begin{equation}
    \sum_{i = 1}^n f^i (X_t) \partial_{a^i} L (X_t, \alpha_t)
    \label{eq:lagrange8}
  \end{equation}
  is constant with respect to time $t \in [t_0, T]$.
\end{theorem}

\begin{proof}
  Let us compute the derivative with respect to time
  of~{\eqref{eq:lagrange8}}, then, by Euler-Lagrange
  equations~{\eqref{eq:EulerLagrange}}, we have
  \begin{equation*} 
  \begin{split}
    \frac{\mathd}{\mathd t} \left( \sum_{i = 1}^n f^i (X_t) \partial_{a^i} L
    (X_t, \alpha_t) \right) = & \, \sum_{i,j = 1}^n \partial_{x^j} f^i (X_t)
    \partial_{a^i} L (X_t, \alpha_t) \frac{\mathd X^j_t}{\mathd t}
    + \sum_{i = 1}^n f^i (X_t) \frac{\mathd}{\mathd t}
    (\partial_{a^i} L (X_t, \alpha_t)) \\
    = & \, \sum_{i,j = 1}^n [\partial_{x^j} f^i (X_t) \alpha^j_t
    \partial_{a^i} L (X_t, \alpha_t) + f^i (X_t) (\partial_{x^i} L (X_t,
    \alpha_t))],
  \end{split}
  \end{equation*} 
  which is zero as a consequence of equation~{\eqref{eq:lagrange4}}.
\end{proof}

It is possible to give an equivalent formulation of
Theorem~\ref{theorem:cNoether1} using the Lie point symmetries of
Hamilton-Jacobi equation.

\begin{theorem}[Noether theorem, Hamilton-Jacobi version]
  \label{theorem:cNoether2}Let $\Omega (x,u_x) = \sum_{i = 1}^n f^i (x) u_{x^i}$ be a
  contact symmetry of the Hamilton-Jacobi equation~{\eqref{eq:lagrange3}}.
  Then, if $U \in C^{1, 2} ([t_{0,}, T] \times \mathbb{R}^n, \mathbb{R})$ is a
  solution to equation~{\eqref{eq:lagrange3}}, we have that
  \begin{equation}
    \Omega (X_t, \nabla U (X_t)) = \sum_{i = 1}^n f^i (X_t) \partial_{x^i} U
    (X_t), \label{eq:lagrange6}
  \end{equation}
  where $X_t$ is the solution to equation~{\eqref{eq:lagrange1}} with
  $\alpha^i_t = \mathcal{A}^i (X_t, \nabla U (X_t))$ (see Section \ref{s:opt-ctrl} for the definition of the map $\mathcal{A}$), is constant with respect to
  time $t \in [t_0, T]$.
\end{theorem}

\begin{lemma}
  $Y_{\Omega}$ is a contact symmetry of the Hamilton-Jacobi equation~{\eqref{eq:lagrange3}} if and only~if
  \begin{equation}
    \sum_{i = 1}^n \left( \partial_{x^i} \Omega \, \partial_{u_{x^i}} H -
    \partial_{u_{x^i}} \Omega \, \partial_{x^i} H \right) = 0.
    \label{eq:lagrange7}
  \end{equation}
\end{lemma}

\begin{proof}
  It is a consequence of equation~{\eqref{eq:contact1}} and
  Definition~\ref{definition_symmetry}.
  See, e.g., Section~21.2 in~\cite{stephani_differential_1989}.
\end{proof}

\begin{proof}[Proof of Theorem~\ref{theorem:cNoether2}]
	See the proof of Theorem~\ref{theorem_noether} below, where the statement is
	proven in the general stochastic case.
\end{proof}

\begin{remark}
  The two formulations of Noether theorem given by Theorem~\ref{theorem:cNoether1} and
  Theorem~\ref{theorem:cNoether2} are equivalent in the
  sense that $Y^{x, a} = \sum_{i, j = 1}^n ( f^i (x) \partial_{x^i} +
  \partial_{x^j} f^i (x, a) a^j \partial_{a^i}) $ is a symmetry of the optimal
  control problem if and only if $\Omega$ is a contact symmetry of the related
  Hamilton-Jacobi equation, namely, equation~{\eqref{eq:lagrange7}} holds.
  Furthermore, if we choose the optimal control $\alpha^i_t$ to be equal to
  $A^i (X_t, \nabla U (X_t))$, then the two conserved
  quantities~{\eqref{eq:lagrange8}} and~{\eqref{eq:lagrange6}} are equal.
\end{remark}

\section{Noether theorem for stochastic optimal control}\label{s:noether} 

\subsection{The case of deterministic HJB equation}

Considering $M = \mathbb{R}_+ \times \mathbb{R}^n$ and denoting the first variable by~$t$ and the other independent variables by~$x^i$, for $i = 1, \ldots, n$, for the Hamilton-Jacobi-Bellman equation we have that $\Delta_\mathcal{E}$ is described by the equation
\begin{equation}
  u_t + \max_{a \in
  K} \left\{ \frac{1}{2}  \sum_{i, j = 1}^n \eta^{i j} (t, x,
  a) u_{x^i x^j} + \sum_{i = 1}^n \mu^i (t, x, a) u_{x^i} + L (t, x, a)
  \right\} = 0. \label{eq:HJB}
\end{equation}

Equation~{\eqref{eq:HJB}} is a special kind of evolution equation since it has
the form
\begin{equation}
  u_t + H (t, x, u, u_x, u_{x, x})=0,
  \label{eq:evolution}
\end{equation}
for some smooth function $H \in C^2 (\mathbb{R} \times J^2 (\mathbb{R}^n,
\mathbb{R}))$, where $u_{x}=(u_{x_1},\ldots,u_{x_n})$, and $u_{x x} = (u_{x^i x^j})_{i,j=1,\ldots,n} $. 
In this case, it is convenient to choose a generating function of the form
\begin{equation}
  \Omega (t, x, u, u_{x}). \label{eq:generatingevolution}
\end{equation}

\begin{remark}\label{remark:time}
It is important to notice that, for a generic contact symmetry on $J^2(M,\mathbb{R})=J^2(\mathbb{R}_+ \times \mathbb{R}^n,\mathbb{R})$, the generating function has the form 
\begin{equation}\label{eq:generatingevolution2}
\tilde{\Omega}(t,x,u,u_t,u_x),
\end{equation}
depending also on the variable $u_t$ which represents the time derivative. 
Choosing a generating function of the form~\eqref{eq:generatingevolution} instead of the form~\eqref{eq:generatingevolution2} means to \emph{consider contact transformations that do not change the time variable~$t$}. 
The main reason is that the time variable in stochastic equations plays a peculiar role and cannot be changed in the same way as the spacial variable~$x$. 
Nevertheless, in~\cite{lescot_isovectors_2004,Lescot_Zambrini2008,thieullen_symmetries_1997} also a special kind of time change has been considered corresponding to the generating function 
\begin{equation}\label{eq:timetransform}
\tilde{\Omega}=f(t)u_t+ \Omega_{\text{Lie},f,g} (t, x, u, u_{x}),
\end{equation}
where $f \in C^{\infty}(\mathbb{R}_+,\mathbb{R})$, and $ \Omega_{\text{Lie},f,g} (t, x, u, u_{x})$ is the generator of a Lie point transformation, see equation~\eqref{eq:Liegenerator} (see also Remark~\ref{remark:further} for a further discussion of this point).
\end{remark}

\begin{theorem}
  Consider an evolution PDE of the form~{\eqref{eq:evolution}}. An infinitesimal contact transformation generated by the function $\Omega$ of the
  form~{\eqref{eq:generatingevolution}} is a {\emph{contact symmetry for
  equation~{\eqref{eq:evolution}}}} if and only~if
  \begin{equation}
    \partial_t \Omega - H \partial_u \Omega + \sum_{i, j = 1}^n \left( \mathcal{D}_{x^i} \Omega \partial_{u_{x^i}} H + \mathcal{D}_{x^i x^j} \Omega \partial_{u_{x^i
    x^j}} H - \mathcal{D}_{x^i} H \partial_{u_{x^i}} \Omega \right) = 0, \label{eq:Omega}
  \end{equation}
  where $\mathcal{D}_{x^i}$ are defined in equation~\eqref{eq:mathcalD} and $\mathcal{D}_{x^i x^j} \cdot = \mathcal{D}_{x^i} (\mathcal{D}_{x^j} (\cdot))$.
\end{theorem}

\begin{proof}
  The statement follows directly from Theorem~\ref{thm:contact1} and Theorem~\ref{theorem:determining} (in particular equations~\eqref{eq:contact1} and~\eqref{eq:symmetry}).
\end{proof}

Let us introduce
\[ O_t = \Omega (t, X_t, U (t,X_t), \nabla U (t,X_t)), \]
where $X_t$ is a solution to equation~{\eqref{eq:sde-socp}} with respect to an
optimal control~$A_t^\ast$.

\begin{assumption}\label{ass:max1}
	There exists at least one measurable
	function $\mathcal{A} (t, x, u_{x}, u_{x x})$ such that
	\[ \mathcal{A} (t,x, u_{x}, u_{x x}) \in \arg \max \mathcal{H} (t,x,
	u_{x}, u_{x x}, \cdot), \]
	where
	\[ \mathcal{H} (t,x, u_{x}, u_{x x}, a) = \sum_{i = 1}^n \mu^i (t,x, a)
	u_{x^i} + \frac{1}{2} \sum_{i, j = 1}^n \sum_{\ell = 1}^m
	\sigma^i_{\ell} (t,x, a) \sigma^j_{\ell} (t,x, a) u_{x^i x^j} + L (t,x, a) .
	\]
\end{assumption}
As a consequence of Assumption~\ref{ass:max1}, we can choose the process $\alpha_t = \mathcal{A} (t,X_t, \nabla U (t,X_t), D^2 U (t,X_t))$ to be the optimal control provided that the solution $U$ to equation~{\eqref{eq:hjbH}} is
at least~$C^2$.

The next result is our first stochastic generalization of Noether theorem.

\begin{theorem}
  \label{theorem_noether}
  Let Assumption~\ref{ass:max1} hold true.
  Suppose that the solution $U$ to
  equation~{\eqref{eq:hjbH}} is continuously differentiable with respect to
  time and $C^2$ with respect to~$x$. If {$\Omega$} is a contact symmetry of equation~{\eqref{eq:hjbH}},
  then $O_t$ is a local martingale.
\end{theorem}

\begin{remark}\label{remark:further}
The works~\cite{lescot_isovectors_2004,Lescot_Zambrini2008,thieullen_symmetries_1997} present a Noether theorem involving a time change and a Lie point transformation with a generator of the form~\eqref{eq:timetransform} for an optimal control system with affine type control and an objective function with quadratic dependence from the control.  
More precisely, they proved that, if $\tilde{\Omega}$ of the form \eqref{eq:timetransform} is a symmetry of the HJB equation, then the process
\begin{equation}\label{eq:conservationtime}
	\begin{split}
	\hat{O}_t& = \tilde{\Omega}(t,X_t,U(t,X_t),\partial_tU(t,X_t),\nabla U(t,X_t))\\
	& = -f(t)H(t,\nabla U(t,X_t), D^2 U(t,X_t))+ \Omega_{\text{Lie},f,g} (t,X_t,U(t,X_t),\nabla U(t,X_t)), 
	\end{split} 
\end{equation}
is a local martingale. 
The presence of some time invariance was essential in the papers \cite{arnaudon_stochastic_2017,privault_stochastic_2010} for extending the concept of integrable systems to the stochastic framework. 
We expect that the martingality of the process~\eqref{eq:conservationtime} holds also in the general setting presented here. 
Since it is not completely clear what it is the role of time change in our setting and if the conservation of~\eqref{eq:conservationtime} holds for more general time changes, we prefer to postpone this analysis to some later works.
\end{remark}

From now on we take $H$ as in Section~\ref{s:det-hjb}, namely, 
\[ H (t,x, u_{x}, u_{x x}) = \sup_{a \in K } \mathcal{H} (t,x, u_{x}, u_{x x}, a). \]
In order to prove Theorem~\ref{theorem_noether}, we anticipate the following result.
\begin{lemma}
  \label{lemma_main}
  We have that
  \begin{align*}
    \partial_{u_{x^i}} H & = \mu^i (t,x, \mathcal{A} (t,x, u_{x}, u_{x
    x})), \\
    \partial_{u_{x^i x^j}} H & = \frac{1}{2} \sum_{\ell=1}^m \sigma^i_{\ell}
    (t,x, \mathcal{A} (t,x, u_{x}, u_{x x})) \sigma^j_{\ell} (t,x, \mathcal{A}
    (t,x, u_{x}, u_{x x})) . 
  \end{align*}
\end{lemma}

\begin{proof}
  In the case where $\mu$, $\sigma$, and $\mathcal{A}$ are $C^1$ in all their variables, the result follows from the fact that 
  \[ \partial_{a^i} \mathcal{H}(t,x,u_x,u_{xx},\mathcal{A}(t,x,u_x,u_{x x})) =0. \]
  The general case is a consequence of Assumption~\ref{ass:max1} and the Envelope Theorem.
  For the latter we refer the reader to, e.g.,~{\cite{deCarvalho_general_2009,milgrom_envelope_2002}}.
\end{proof}

\begin{proof}[Proof of Theorem~\ref{theorem_noether}]
  We compute the differential of $O_t$ using It{\^o} formula, to get
  \begin{align*}
    \mathd O_t = & \, \mathd \Omega (t, X_t, U (t,X_t), \nabla U (t,X_t))
    \\
    = & \, \partial_t \Omega (t, X_t, U (t,X_t), \nabla U (t,X_t)) \mathd
    t + \partial_u \Omega (t, X_t, U (t,X_t), \nabla U (t,X_t)) \mathd U
    (t, X_t) \\
    & + \sum_{i = 1}^n \partial_{x^i} \Omega (t, X_t, U (t,X_t), \nabla U (t,X_t)) \mathd X^i_t \\
    & + \sum_{i = 1}^n \partial_{u_{x^i}} \Omega (t, X_t, U (t,X_t), \nabla U (t,X_t)) \mathd \partial_{x^i} U (t, X_t) \\
    & + \frac{1}{2} \sum_{i, j = 1}^n \partial_{x^i x^j} \Omega (t, X_t, U (t,X_t), \nabla U (t,X_t)) \mathd [X^i, X^j]_t \\
    & + \frac{1}{2} \partial_{u u} \Omega (t, X_t, U (t,X_t), \nabla U (t,X_t)) \mathd [U (\cdot,X_{\cdot} ), U (\cdot,X_{\cdot})]_t
    \\
    & + \frac{1}{2} \sum_{j = 1}^n \partial_{u x^j} \Omega (t, X_t, U (t,X_t), \nabla U (t,X_t)) \mathd [U (\cdot, X_{\cdot}), X^j]_t
    \\
    & + \frac{1}{2} \sum_{j = 1}^n \partial_{u u_{x^j}} \Omega (t, X_t, U (t,X_t), \nabla U (t,X_t)) \mathd [\partial_{x^j} U (\cdot,X_{\cdot}), U (\cdot,X_{\cdot})]_t \\
    & + \frac{1}{2} \sum_{i, j = 1}^n \partial_{u_{x^i} x^j} \Omega (t, X_t, U (t,X_t), \nabla U (t,X_t)) \mathd [\partial_{x^i} U (\cdot,X_{\cdot}), X^j]_t \\
    & + \frac{1}{2} \sum_{i, j = 1}^n \partial_{u_{x^i} u_{x^j}} \Omega (t, X_t, U (t,X_t), \nabla U (t,X_t)) \mathd [\partial_{x^i} U (\cdot,X_{\cdot}),
    \partial_{x^j} U (\cdot,X_{\cdot})]_t .
  \end{align*}
  Since $U \in C^{2, 3} ([0, T] \times \mathbb{R}^n, \mathbb{R})$, we
  also have
  \begin{align}
    \mathd U (t,X_t) & = \partial_t U (t,X_t) \mathd t + \sum_{i = 1}^n
    \partial_{x^i} U (t,X_t) \mathd X^i_t + \frac{1}{2} \sum_{i, j = 1}^n
    \partial_{x^i x^j} U (t,X_t) \mathd [X^i, X^j]_t, \label{eq:O2} \\
    \mathd \partial_{x^i} U (t,X_t) & = \partial_{x^i, t} U (t,X_t) \mathd
    t + \sum_{j = 1}^n \partial_{x^i x^j} U (t,X_t) \mathd X^j_t +
    \frac{1}{2} \sum_{j, k = 1}^n \partial_{x^i x^j x^k} U (t,X_t) \mathd
    [X^j, X^k]_t . \label{eq:O3} 
  \end{align}
  Exploiting equations~{\eqref{eq:O2}} and~{\eqref{eq:O3}}, the fact that $X_t$ is solution to~{\eqref{eq:sde-socp}}, and the relations
  \begin{align*}
    \mathd [X^i, X^j]_t = & \sum_{\ell = 1}^m \sigma^i_{\ell} (t,X_t,
    \alpha_t) \sigma^j_{\ell} (t,X_t, \alpha_t) \mathd t, \\
    \mathd [U (\cdot,X_{\cdot}), X^i]_t = & \sum_{j = 1}^n \sum_{\ell =
    1}^m \partial_{x^j} U (t,X_t) \sigma^j_{\ell} (t,X_t, \alpha_t)
    \sigma^i_{\ell} (t,X_t, \alpha_t) \mathd t, \\
    \mathd [U (\cdot,X_{\cdot}), U (\cdot,X_{\cdot})]_t = & \sum_{i, j =
    1}^n \sum_{\ell = 1}^m \partial_{x^j} U (t,X_t) \partial_{x^i} U (t,X_t) \sigma^j_{\ell} (t,X_t, \alpha_t) \sigma^i_{\ell} (t,X_t, \alpha_t)
    \mathd t, \\
    \mathd [U (\cdot,X_{\cdot}), \partial_{x^i} U (\cdot,X_{\cdot})]_t = &
    \sum_{k, j = 1}^n \sum_{\ell = 1}^m \partial_{x^j} U (t,X_t)
    \partial_{x^i x^k} U (t,X_t) \sigma^j_{\ell} (t,X_t, \alpha_t)
    \sigma^k_{\ell} (t,X_t, \alpha_t) \mathd t, \\
    \mathd [\partial_{x^{l}} U (\cdot,X_{\cdot}), \partial_{x^i} U (\cdot,X_{\cdot})]_t = & \sum_{k, j = 1}^n \sum_{\ell = 1}^m
    \partial_{x^l x^j} U (t,X_t) \partial_{x^i x^k} U (t,X_t)
    \sigma^j_{\ell} (t,X_t, \alpha_t) \sigma^i_{\ell} (t,X_t, \alpha_t) \mathd
    t, 
  \end{align*}
  we obtain
  \begin{equation*}
  \begin{split} 
    \mathd O_t = & \, \sum_{i, k = 1}^n \mu^i (t,X_t, \alpha_t) \left( \partial_{x^i}
    \Omega + u_{x^i} \partial_u \Omega + u_{x^i x^k} \partial_{u_{x^k}} \Omega
    \right) (t, X_t, U (t,X_t), \nabla U (t,X_t), D^2 U (t,X_t))
    \mathd t \\
    & + \frac{1}{2} \sum_{i, j, k, l = 1}^n \sum_{\ell = 1}^m
    \sigma^i_{\ell} (t,X_t, \alpha_t) \sigma^j_{\ell} (t,X_t, \alpha_t) \biggl(
    \partial_{x^i, x^j} \Omega + u_{x^j} \partial_{x^i,
    u} \Omega + u_{x^j x^k} \partial_{x^i u_{x^k}} \Omega + u_{x^i} \partial_{x^j, u} \Omega  \\
    &  + u_{x^i} u_{x^j} \partial_{u u}
    \Omega + u_{x^i} u_{x^j x^k} \partial_{u u_{x^k}} \Omega + u_{x^i x^j}
    \partial_u \Omega + u_{x^i x^k} \partial_{x^j u_{x^k}} \Omega + u_{x^i
    x^k} u_{x^j} \partial_{u_{x^k} u} \Omega \\
    & + u_{x^i x^k} u_{x^j x^{l}}
    \partial_{u_{x^k} u_{x^{l}}} \Omega + u_{x^i x^j x^k}
    \partial_{u_{x^k}} \Omega \biggr) (t, X_t, U (t,X_t), \nabla U (t,X_t), D^2 U (t,X_t), D^3 U (t,X_t)) \mathd t \\
    & + \left( \partial_t U \, \partial_u \Omega +
    \sum_{i = 1}^n \partial_{t x^i}U \, \partial_{u_{x^i}} \Omega + \partial_t \Omega \right) (t, X_t, U (t,X_t), \nabla U (t,X_t), D^2 U (t,X_t)) \mathd t + \mathd
    M_t ,
  \end{split} 
  \end{equation*}
  where $M_t$ is a local martingale. Using the explicit definition of
  $\mathcal{D}_{x^i}$, it is simple to note that
  \begin{align*}
    \mathcal{D}_{x^i} \Omega = & \, \partial_{x^i} \Omega + u_{x^i}
    \partial_u \Omega + \sum_{k = 1}^n u_{x^i x^k} \partial_{u_{x^k}} \Omega, \\
    \mathcal{D}_{x^i x^j} \Omega = & \, \partial_{x^i x^j} \Omega + u_{x^j} \partial_{x^i u} \Omega + u_{x^i} \partial_{x^j u} \Omega + u_{x^i} u_{x^j}
    \partial_{u u} \Omega + u_{x^i x^j} \partial_u \Omega 
    \\
    & + \sum_{k, l = 1}^n \Bigl( u_{x^j x^k} \partial_{x^i
   	u_{x^k}} \Omega 
    + u_{x^i} u_{x^j x^k} \partial_{u u_{x^k}} \Omega + u_{x^i x^k} \partial_{x^j
    u_{x^k}} \Omega \\ 
	& + u_{x^i x^k} u_{x^j} \partial_{u_{x^k} u} \Omega +
    u_{x^i x^k} u_{x^j x^{l}} \partial_{u_{x^k} u_{x^{l}}} \Omega +
    u_{x^i x^j x^k} \partial_{u_{x^k}} \Omega \Bigr), 
  \end{align*}
  and we have
  \[ \partial_t U = - H(t,x,\nabla U,D^2 U) \qquad \text{and} \qquad \partial_{t, x^i} U = - (\mathcal{D}_{x^i} H) (t,x,\nabla U,D^2 U, D^3U) . \]
  Using Lemma~\ref{lemma_main}, the fact that we can choose $\alpha_t = \mathcal{A}
  (t,X_t, \nabla U (t,X_t), D^2 U (t,X_t))$, and
  the determining equations~{\eqref{eq:Omega}}, we obtain
  \begin{equation*}
  \begin{split} 
    & \mathd O_t = \\
    = & \, \sum_{i = 1}^n \mu^i (t,X_t, \alpha_t) (\mathcal{D}_{x^i} \Omega) (t, X_t,
    U (t,X_t), \nabla U (t,X_t), D^2 U (t,X_t)) \mathd t \\
    & +
    \frac{1}{2} \sum_{\ell = 1}^m \sum_{i, j = 1}^n \sigma^i_{\ell} (t,X_t,
    \alpha_t) \cdot \sigma^j_{\ell} (t,X_t,\alpha_t) (\mathcal{D}_{x^i x^j} \Omega) (t, X_t, U (t,X_t), \nabla U (t,X_t), D^2 U (t,X_t), D^3 U (t,X_t))
    \mathd t \\
    & + \sum_{i = 1}^n \left( - H \partial_u \Omega - \mathcal{D}_{x^i} H
    \partial_{u_{x^i}} \Omega + \partial_t \Omega \right) (t, X_t, U (t,X_t),
    \nabla U (t,X_t), D^2 U (t,X_t)) \mathd t + \mathd M_t
    \\
    = &  \sum_{i, j = 1}^n \left( \mathcal{D}_{x^i} \Omega \partial_{u_{x^i}} H +
    \mathcal{D}_{x^i x^j} \Omega \partial_{u_{x^i x^j}} H - H \partial_u \Omega -
    \mathcal{D}_{x^i} H \partial_{u_{x^i}} \Omega + \partial_t \Omega \right)
    \\
    & (t, X_t,U (t,X_t), \nabla U (t,X_t), D^2 U (t,X_t),
    D^3 U (t,X_t)) \mathd t + \mathd M_t \\
    = & \, \mathd M_t, 
  \end{split} 
  \end{equation*}
  which concludes the proof.
\end{proof}

\subsection{The case of stochastic HJB equation}

We face the problem of stochastic HJB equation, that is, we consider, as in Section~\ref{s:stoch-hjb}, 
\begin{equation}\label{eq:HtoS}
\begin{split} 
  \mathcal{H}^S (t, x, u_{x}, u_{x x}, a, \psi_{x}) = & \, \sum_{i = 1}^n
  \mu^i (t, x, a, \omega) u_{x^i} + \frac{1}{2} \sum_{i, j = 1}^n \sum_{\ell =
  1}^m \sigma^i_{\ell} (t, x, a, \omega) \sigma^j_{\ell} (t, x, a, \omega)
  u_{x^i x^j} \\
  & + \sum_{i = 1}^n \sum_{\ell = 1}^m \sigma_{\ell}^i (t, x, a, \omega)
  \psi^{\ell}_{x^i} + L^S (t, x, a) .
\end{split} 
\end{equation}
and
\[ H^S (t, x, u_{x}, u_{x x}, \psi_{x}) = \sup_{a \in K} \mathcal{H}^S
   (t, x, u_{x}, u_{x x}, a, \psi_{x}) . \]
In this case,
\begin{equation}
  \mathd U_t (x) = - H^S (t, x, \nabla U, D^2 U,
  \nabla \Psi) \mathd t + \sum_{\ell = 1}^m \Psi^{\ell}_t
  (x) \mathd W^{\ell}_t . \label{eq:dU}
\end{equation}
Though some ideas concerning symmetries for SPDEs are discussed, e.g., in~\cite{delara_reduction_1998} and~\cite{devecchi_finite_2017}, a general theory has not been developed yet. For this reason, we extend the notion of infinitesimal symmetry introduced in Definition~\ref{definition_symmetry} in the following way. 
Hereafter, we consider the probability space $(\mathcal{W},\mathcal{F}_t,\mathbb{P})$ where $\mathcal{W}=C^0(\mathbb{R},\mathbb{R}^m)$ is the canonical space for the Brownian motion $W$, $\mathcal{F}_t$ is the natural filtration generated by $W_t$, and $\mathbb{P}$ is the Wiener measure on $\mathcal{W}$. 
\begin{definition}\label{def:contact-symmetry-for-dU}
Let $\Omega\colon \mathbb{R}_+\times J^1(\mathbb{R}^n,\mathbb{R}) \times \mathcal{W} \rightarrow \mathbb{R}$  be a predictable regular random field on $\mathbb{R}_+ \times J^1(\mathbb{R}^n,\mathbb{R})$ which is $C^1$ with respect to the time $t$ and $C^2$ in all other variables. We say that $Y_{\Omega}$ is a {\emph{contact symmetry for
  equation~{\eqref{eq:dU}}}} when we have
  \[ \partial_t \Omega - H^S \partial_u \Omega + \sum_{i,j=1}^n \left( \mathcal{D}_{x^i} \Omega \partial_{u_{x^i}} H^S + \mathcal{D}_{x^i x^j} \Omega
     \partial_{u_{x^i x^j}} H^S - \mathcal{D}_{x^i} H^S
     \partial_{u_{x^i}} \Omega \right) = 0. \]
\end{definition}

\begin{assumption}\label{ass:max2}
	There exists at least one measurable
	function $\mathcal{A}^S (t, x, u_{x}, u_{x x}, \psi_{x})$ such that
	\[ \mathcal{A}^S (t, x, u_{x}, u_{x x}, \psi_{x}) \in \arg \max
	\mathcal{H}^S (t, x, u_{x}, u_{x x}, \cdot, \psi_{x}), \]
	where $\mathcal{H}^S$ is defined by equation~{\eqref{eq:HtoS}}.
\end{assumption}

\begin{lemma}
  \label{lemma_main2}
  We have that
  \begin{align*}
    \partial_{u_{x^i}} H^S = \, & \mu^i (t, x, \mathcal{A}^S (t, x, u_{x}, u_{x
    x}, \psi_{x}), \omega), \\
    \partial_{u_{x^i x^j}} H^S = \, & \frac{1}{2} \sum_{\ell=1}^m \sigma^i_{\ell}
    (t, x, \mathcal{A}^S (x, u_{x}, u_{x x}, \psi_{x}), \omega)
    \, \sigma^j_{\ell} (t, x, \mathcal{A}^S (t, x, u_{x}, u_{x x},
    \psi_{x}), \omega), \\
    \partial_{\psi^{\ell}_{x^i}} H^S = \, & \sigma_{\ell}^i (t, x, \mathcal{A}^S
    (t, x, u_{x}, u_{x x}, \psi_{x}), \omega) . 
  \end{align*}
\end{lemma}

\begin{proof}
  The proof is similar to the one of Lemma~\ref{lemma_main}.
\end{proof}

The following result represents our second stochastic generalization of Noether theorem.

\begin{theorem}
  \label{theorem_noether2}
  Let Assumption~\ref{ass:max2} hold true. 
  Suppose that the solution $(U, \Psi)$ to
  equation~{\eqref{eq:dU}} is continuously differentiable with respect to
  time and $C^3$ with respect to $x$ almost surely.
  If {$\Omega$} is a contact symmetry of
  equation~{\eqref{eq:hjbH}}, then
  \begin{equation}\label{eq:tildeO}
  \begin{split} 
  \tilde{O}_t = & \, \Omega (t, X_t, U (t, X_t), \nabla  U (t, X_t) )  \\
  & - \frac{1}{2}
  \int_0^t \sum_{i,j=1}^n \sum_{\ell=1}^m \Bigl( \partial_{u u} \Omega\left(
  (\Psi^{\ell}_s)^2 +  2 u_{x^i} \sigma^i_{\ell} \Psi^{\ell}_s \right)
  - \partial_{x^i u} \Omega \sigma^i_{\ell} \Psi^{\ell} - \partial_{u_{x^i}} \Omega \partial_{x^i} \sigma^j_{\ell}
  \Psi^{\ell}_{x^j}  \\
  & + \partial_{x^i u_{x^j}} \Omega
  \sigma^i_{\ell} \Psi^{\ell}_{x^j} 
  + \partial_{u_{x^i} u_{x^j}} \Omega
  (\Psi^{\ell}_{x^i} \Psi_{x^j}^{\ell} + \sigma^i_{\ell} u_{x^i}
  \Psi^{\ell}_{x_j} + \sigma^j_{\ell} u_{x^j} \Psi^{\ell}_{x_i}) 
  \\
  & + \partial_{u u_{x^j}} \Omega \bigl( \Psi^{\ell}
  \Psi^{\ell}_{x^j} + \sigma^i_{\ell} u_{x^i} \Psi^{\ell}_{x^j} +
  \sigma^i_{\ell} u_{x^i x^j} \Psi^{\ell} \bigr) \Bigr) 
  (s, X_s, U (s, X_s), \nabla U (s, X_s))  \mathd s 
  \end{split} 
  \end{equation}
  is a local martingale.
\end{theorem}

\begin{proof}
  Since the proof is similar to the one of Theorem~\ref{theorem_noether}, we report
  here only some steps of the proof. By Theorem~\ref{thm:ito-kunita}, we have
  \begin{equation}\label{eq:ito-U}
  \begin{split} 
    \mathd U (t, X_t) = & - H^S (t, X_t, \nabla U (t, X_t), D^2 U (t, X_t), \nabla \Psi_t (X_t)) \mathd t \\
    & + \sum_{\ell = 1}^m \Psi^{\ell}_t (X_t) \mathd W^{\ell}_t +
    \sum_{\ell = 1}^m \sum_{i = 1}^n \partial_{x^i} \Psi^{\ell} (X_t)
    \mathd [W^{\ell}, X^i]_t \\
    & + \sum_{i = 1}^n \partial_{x^i} U (t, X_t) \mathd X^i_t +
    \frac{1}{2} \sum_{i, j = 1}^n \partial_{x^i x^j} U (t, X_t) \mathd [X^i,
    X^j]_t, 
    \end{split} 
    \end{equation} 
    and
    \begin{equation}\label{eq:ito-partial-U}
    \begin{split} 
    \mathd \partial_{x^k} U (t, X_t) = & - \mathcal{D}_{x^k}H^S (t, X_t, \nabla U (t, X_t), D^2 U (X_t, t), \nabla \Psi_t (X_t)) \mathd t
    \\
    & + \sum_{\ell = 1}^m \partial_{x^k} \Psi^{\ell}_t (X_t) \mathd
    W^{\ell}_t + \sum_{i = 1}^n \sum_{\ell = 1}^m \partial_{x^i x^k}
    \Psi^{\ell}_t (X_t) \mathd [W^{\ell}, X^i]_t \\
    & + \sum_{i = 1}^n \partial_{x^i x^k} U (t, X_t) \mathd X^i_t +
    \frac{1}{2} \sum_{i, j = 1}^n \partial_{x^i x^j x^k} U (t, X_t) \mathd [X^i,
    X^j]_t . 
  \end{split} 
  \end{equation}
  Writing, as usual, ${O}_t = \Omega (t, X_t, U (t, X_t), \nabla U (t, X_t))$ and $\alpha_t = \mathcal{A}^S (t, X_t, \nabla U(t,X_t), D^2 U(t,X_t), \nabla \Psi_t (X_t))$, we
  have
  \[ \begin{split}
       \mathd {O}_t = & \, \sum_{i=1}^n \mu^i (t, X_t, \alpha_t) \partial_{x^i} \Omega (t, X_t, U (t, X_t), \nabla U (t, X_t)) \mathd t\\
       & + \frac{1}{2} \sum_{i,j=1}^n \sum_{\ell=1}^m \sigma^i_{\ell} (t, X_t, \alpha_t)
       \sigma^j_{\ell} (t, X_t, \alpha_t) \partial_{x^i x^j} \Omega (t, X_t, U (t, X_t), \nabla U (t, X_t)) \mathd t\\
       & + \partial_u \Omega \mathd U (t, X_t) + \frac{1}{2} \partial_{u u}
       \Omega \mathd [U, U] + \sum_{i=1}^n \partial_{x^i u} \Omega \mathd [X^i, U] + \sum_{i=1}^n \partial_{u_{x^i}} \Omega \mathd \partial_{x^i}U (t,X_t) \\
       & + \frac{1}{2} \sum_{i,j=1}^n
       \partial_{u_{x^i} u_{x^j}} \Omega \mathd [\partial_{x^i}U,\partial_{x^j}U] +
       \sum_{j=1}^n \partial_{u u_{x^j}} \Omega \mathd [U, \partial_{x^j}U] + \sum_{i,i=1}^n \partial_{x^i u_{x^j}} \Omega \mathd [X^i, \partial_{x^j}U]\\
       & + \partial_t \Omega \mathd t + \mathd \tilde{M}_t.
  \end{split} \]
  Plugging in equations~\eqref{eq:ito-U} and~\eqref{eq:ito-partial-U}, and exploiting Theorem~\ref{thm:ito-kunita} in order to compute the quadratic variations, we get
  \begin{align*} 
       \mathd {O}_t = & \, \sum_{i = 1}^n \mu^i (t, X_t, \alpha_t) \left( 
       \partial_{x^i} \Omega + u_{x^i} \partial_u \Omega + \sum_{k = 1}^n u_{x^i x^k}
       \partial_{u_{x^k}} \Omega \right) (t, X_t, U (t, X_t), \nabla U (t, X_t), D^2 U (t, X_t)) \mathd t\\
       & + \frac{1}{2} \sum_{i, j = 1}^n \sum_{\ell = 1}^m \sigma^i_{\ell}
       (t, X_t, \alpha_t) \sigma^j_{\ell} (t, X_t, \alpha_t) \Biggl( \sum_{k, l = 1}^n
       \partial_{x^i x^j} \Omega + u_{x^j} \partial_{x^i u} \Omega + u_{x^j
       x^k} \partial_{x^i u_{x^k}} \Omega + u_{x^i} \partial_{x^j u} \Omega
       \\
       & + u_{x^i} u_{x^j} \partial_{u, u} \Omega + u_{x^i} u_{x^j
       x^k} \partial_{u u_{x^k}} \Omega + u_{x^i x^j} \partial_u \Omega +
       u_{x^i x^k} \partial_{x^j u_{x^k}} \Omega + u_{x^i x^k} u_{x^j}
       \partial_{u_{x^k} u} \Omega \\
       & + u_{x^i x^k} u_{x^j x^{l}}
       \partial_{u_{x^k} u_{x^{l}}} \Omega + u_{x^i x^j x^k}
       \partial_{u_{x^k}} \Omega \Biggr) (t, X_t, U (t, X_t), \nabla U (t, X_t), D^2 U (t, X_t), D^3 U (t, X_t)) \mathd t \\
       & - \partial_u \Omega (t, X_t, U (t, X_t),
       \nabla U (t, X_t)) H^S (t, X_t, U (t, X_t), \nabla U (t, X_t),
       D^2 U (t, X_t), \nabla \Psi_t (X_t)) \mathd t\\
       & + \frac{1}{2} \partial_{u u} \Omega (t, X_t, U (t, X_t), \nabla U (t, X_t)) \sum_{i = 1}^n \sum_{\ell = 1}^m \left( \Psi^{\ell}_t
       (x)^2 + 2 \partial_{x^i} U (t, x) \sigma^i_{\ell} (x, a)
       \Psi^{\ell}_t (x) \right) (t, X_t, \alpha_t) \mathd t\\
       & + \sum_{i = 1}^n \sum_{\ell = 1}^m \partial_{x^i u} \Omega (t,
       X_t, U (t, X_t), \nabla U (t, X_t)) \sigma^i_{\ell} (X_t, \alpha_t)
       \Psi^{\ell}_t (X_t) \mathd t\\
       & - \sum_{i = 1}^n \left( \partial_{u_{x^i}} \Omega D_{x^i}
       H^S \right) (t, X_t, U (t, X_t), \nabla U (t, X_t), D^2 U (t, X_t), \nabla \Psi_t (X_t)) \mathd t\\
       & + \sum_{i, k = 1}^n \sum_{\ell = 1}^m \left(
       \partial_{u_{x^i}} \Omega \sigma_{\ell}^k \partial_{x^i x^k}
       \Psi_t^{\ell} + \partial_u \Omega \sigma^i_{\ell}
       \Psi^{\ell}_{x^i} \right) (t, X_t, U (t, X_t), \nabla U (t, X_t), \alpha_t) \mathd t\\
       & + \sum_{i, j = 1}^n \sum_{\ell = 1}^m \left( \partial_{u_{x^i}
       u_{x^j}} \Omega (\partial_{x^i} \Psi^{\ell}_t \partial_{x^j}
       \Psi_t^{\ell} + \sigma^i_{\ell} u_{x^i} \partial_{x^j}
       \Psi^{\ell} + \sigma^j_{\ell} u_{x^j} \partial_{x^i} \Psi^{\ell})
       \right)  (t, X_t, U (t, X_t), \nabla U (t, X_t), \alpha_t) \mathd t \\
       & + \sum_{i, j = 1}^n \sum_{\ell = 1}^m \left( 
       \partial_{x^i u_{x^j}} \Omega \sigma^i_{\ell} \Psi^{\ell}_{x^j}
       \right) (t, X_t, U (t, X_t), \nabla U (t, X_t), \alpha_t) \mathd t \\
       & + \sum_{j = 1}^n \sum_{\ell = 1}^m \left( \partial_{u u_{x^j}}
       \Omega (\Psi^{\ell}_t \partial_{x^j} \Psi^{\ell}_t +
       \sigma^i_{\ell} u_{x^i} \partial_{x^j} \Psi^{\ell}_t +
       \sigma^i_{\ell} u_{x^i x^j} \Psi^{\ell}_t) \right) (t, X_t, U (t, X_t), \nabla U (t, X_t), \alpha_t) \mathd t\\
       & + \mathd \tilde{M}_t
     \end{align*} 
  Notice that, by Definition~\ref{def:contact-symmetry-for-dU}, we have
  \[
       0 = \partial_u \Omega H^S + \sum_{i, j = 1}^n  \left( \partial_{u_{x^i}} \Omega
       \mathcal{D}_{x^i} H^S - \partial_{u_{x^i}} H^S \mathcal{D}_{x^i} \Omega - \partial_{u_{x^i}
       u_{x^j}} H^S \mathcal{D}_{x^i x^j} \Omega \right),
  \]
  which, by Lemma~\ref{lemma_main2}, is equivalent~to
  \begin{multline*}
       \partial_u \Omega H^S + \sum_{i, j = 1}^n \left( \partial_{u_{x^i}} \Omega
       \mathcal{D}_{x^i} H^S - \partial_{u_{x^i}} H^S \mathcal{D}_{x^i} \Omega -
       \partial_{u_{x^i} u_{x^j}} H^S \mathcal{D}_{x^i x^j} \Omega \right) =\\
       = \sum_{i, j = 1}^n \sum_{\ell = 1}^{m} \left( \partial_{u_{x^i}} \Omega (\partial_{x^i}
       \sigma^j_{\ell} \Psi^{\ell}_{x^j} + \sigma^j_{\ell}
       \Psi^{\ell}_{x^j x^i}) + \partial_u \Omega (\sigma^i_{\ell}
       \Psi^{\ell}_{x^i}) \right) . 
  \end{multline*}
  Then we obtain 
  \begin{align*} 
  \mathd {O}_t
  = & \, \frac{1}{2} \partial_{u, u} \Omega (t, X_t, U (t, X_t), \nabla U (t, X_t)) \sum_{i = 1}^n \sum_{\ell = 1}^m \left( (\Psi^{\ell}_t)^2 +
  2 u_{x^i} \sigma^i_{\ell} \Psi^{\ell}_t \right) (t, X_t, U (t, X_t), \nabla U (t, X_t)) \mathd t\\
  & + \sum_{i = 1}^n \sum_{\ell = 1}^m \left( \partial_{x^i u} \Omega
  \sigma^i_{\ell} \Psi^{\ell} \right) (t, X_t, U (t, X_t), \nabla U (t, X_t)) \mathd t + \sum_{i, j \equallim 1}^n \sum_{\ell = 1}^m
  \Bigl( \partial_{u_{x^i} u_{x^j}} \Omega \, \times
  \\
  & \times (\partial_{x^i} \Psi^{\ell}_t
  \partial_{x^j} \Psi_t^{\ell} + \sigma^i_{\ell} u_{x^i}
  \partial_{x^j} \Psi^{\ell}_t + \sigma^j_{\ell} u_{x^j}
  \partial_{x^i} \Psi^{\ell}_t) \Bigr) (t, X_t, U (t, X_t),
  \nabla U (t, X_t), \alpha_t) \mathd t\\
  & + \sum_{i, j = 1}^n \sum_{\ell = 1}^m \left( - \partial_{u_{x^i}}
  \Omega \partial_{x^i} \sigma^j_{\ell} \partial_{x^j} \Psi^{\ell}_t
  + \partial_{x^i u_{x^j}} \Omega \sigma^i_{\ell} \partial_{x^j}
  \Psi^{\ell}_t \right) (t, X_t, U (t, X_t), \nabla U (t, X_t),
  \alpha_t) \mathd t\\
  & + \sum_{i, j = 1}^n \sum_{\ell = 1}^m \left( \partial_{u u_{x^j}}
  \Omega (\Psi^{\ell} \partial_{x^j} \Psi^{\ell}_t +
  \sigma^i_{\ell} u_{x^i} \partial_{x^j} \Psi^{\ell}_t +
  \sigma^i_{\ell} u_{x^i x^j} \Psi^{\ell}_t) \right)(t, X_t, U (t, X_t), \nabla U (t, X_t), \alpha_t) \mathd t \\
  & +
  \mathd \tilde{M}_t.
  \end{align*}
  Following then the same steps as in the proof of Theorem~\ref{theorem_noether} we get the result.
\end{proof}

\begin{corollary}\label{cor:noether2}
  Suppose that $\Omega$ is a Lie point symmetry of the form 
  \[ \Omega (t,x, u, u_x) = c u + g (t,x) - \sum_{k = 1}^n f^k (t,x) u_{x^k}, \]
  where $c \in \mathbb{R}$ and $f^k, g \colon \mathbb{R}^{n+1} \rightarrow \mathbb{R}$
  are smooth functions such that, for $j = 1, \ldots, n$ and $\ell=1,\ldots,m$, 
  \[ \sum_{k = 1}^n \left( f^k \partial_{x^k} \sigma_{\ell}^j - \sigma_{\ell}^k
     \partial_{x^k} f^j \right) = 0. \]
  Then $O_t = \Omega (t, X_t, \nabla U (t, X_t))$ is a local martingale.
\end{corollary}

\begin{proof}
  Under the previous conditions, we have
  \[ \sum_{j = 1}^n \left( - \partial_{u_{x^j}} \Omega \partial_{x^j}
     \sigma^i_{\ell} \right) + \partial_{x^i u_{x^j}} \Omega
     \sigma^i_{\ell} = \partial_{x^i u} \Omega \sigma^i_{\ell} = 0. \]
  The thesis follows from Theorem~\ref{theorem_noether2}.
\end{proof}

\section{Merton's optimal portfolio problem}\label{s:merton} 

In this section, we propose a symmetry analysis of Merton's problem of optimal portfolio selection (see the original paper~\cite{Merton1969lifetime} and~\cite{rogers_optimal_2013} for a review on the subject).
Let us consider a set of controls $ \alpha_t=(c(t),\gamma(t))$ and a controlled diffusion dynamics described by the~SDE 
\begin{equation}\label{eq:mainMerton1}
	\mathd X_t=\bigl((\gamma(t)(\mu(t)-r)+r) X_t -c(t)\bigr) \mathd t + X_t \gamma(t) \sigma(t) \mathd W_t,
\end{equation} 
where $X$ is the wealth process controlled by the proportion $\gamma(t) \in [0,1]$ invested in the risky asset at time $t$ and by the consumption $c(t)\in [0,+\infty)$ per unit time at time~$t$. Moreover, $r$ is the constant interest rate, and $\mu(t), \sigma(t) >0$ are continuous functions such that $\sigma(t)>\epsilon>0$ (or in the case of Subsection \ref{subsection:nonmarkovian} are general continuous predictable stochastic processes).
Fixing some finite time horizon $ T>0 $, the problem of choosing optimal portfolio selection consists in maximizing the objective functional
\begin{equation*}
	\mathbb{E}\left[\int_t^TL(t,\alpha_t) \mathd s+ g(X_T) \right],
\end{equation*}
where
\[
L(t,\alpha_t)= e^{-\rho t} V(c(t)).
\]
Here, $\rho \in (0,+\infty)$ is the discount rate, $V$ is a strictly concave utility function and $g$ is a given function.

Let us remark that the set $K$ introduced in Section~\ref{s:det-hjb} here has the form
\[ K = [0,+\infty) \times [0,1]. \]

\subsection{Markovian case}\label{subsection:Markovian}

The maximization problem introduced above is a particular case  of the general one studied in Section~\ref{s:det-hjb}. 
The associated value function~is 
\[
U(t,x)=\max_{\alpha \in \mathcal{K}_L}\mathbb{E}\left[\int_t^T L(s,\alpha_s) \mathd s+ g(X_T) \, \bigg\vert \, X_t = x \right],
\]
while the HJB equation becomes
\begin{equation*}
\partial_{t}U+\max_{(c,\gamma) \in K}\left\{\mathcal{H}(t,x,\nabla U,D^2U ,(c,\gamma) )\right\}=0,
\end{equation*}
with
\[
\mathcal{H}(t,x,u_x,u_{xx},(c,\gamma))=\exp(-\rho t)V(c)+u_x(\gamma (\mu(t)-r)+r)x-u_x c+\frac{1}{2}u_{xx} \sigma(t)^2 \gamma^2 x^2.
\]

The optimal value $(c^\star,\gamma^\star)$ of $(c,\gamma)$ is given by the solutions to the system
\begin{align*}
\partial_{c}\mathcal{H}=\exp(-\rho t)V^\prime(c)-u_x & =0, \\
\partial_{\gamma}\mathcal{H}=(\mu(t)-r)xu_x+u_{xx}\sigma^2(t) x^2 \gamma & =0,
\end{align*}
that is
\begin{align}
c^\star (t) & = (V^\prime)^{-1}(u_x\exp(\rho t)),\label{eq:cstar}\\
\gamma^\star (t) & =- \frac{(\mu(t)-r)u_x}{x u_{xx} \sigma^2(t)}.\label{eq:gammastar}
\end{align}
The corresponding functional   $\mathcal{H}  $ takes the form
\begin{equation*}
\begin{split} 
H(t,x,u_x,u_{xx})= \, & \mathcal{H}(t,x,u_x,u_{xx},(c^\star(t),\gamma^\star(t)) ) \\
= \, & \exp(-\rho t) V(c^\star(t)) + u_x(\gamma^\star(t)(\mu(t)-r)+r) x-u_x c^\star(t)
+\frac{1}{2}u_{xx} \sigma^2(t) \gamma^\star(t)^2 x^2 \\
= \, &\exp(-\rho t)V\left( (V^\prime)^{-1}(u_x\exp(\rho t)) \right) - \left(\frac{(\mu(t)-r)u_x}{x u_{xx} \sigma^2(t)} (\mu(t)-r)+r \right) xu_x \\
&-u_x(V^\prime)^{-1}(u_x\exp(\rho t))    +\frac{1}{2}u_{xx} \sigma^2 \frac{(\mu(t)-r)^2u^2_x}{x^2 u^2_{xx} \sigma^4(t)}   x^2 \\
= \, &\exp(-\rho t)V\left((V^\prime)^{-1}(u_x\exp(\rho t))\right) - \left( \frac{(\mu(t)-r)u_x}{x u_{xx} \sigma^2(t)} (\mu(t)-r)+r \right) xu_x \\
&-u_x(V^\prime)^{-1}(u_x\exp(\rho t))    +\frac{1}{2} \frac{(\mu(t)-r)^2u^2_x}{ u_{xx} \sigma^2(t)} .
\end{split} 
\end{equation*}
So we study the following PDE
\begin{equation}
u_t -\frac{\delta(t)}{2}  \frac{u_x^2}{u_{x x}} + K (t,x,u_x) = 0, \label{eq:main}
\end{equation}
with 
\begin{align}
K(t,x,u_x) = \, & h_V (t,u_x)+ r x u_x , \label{eq:MertonK}\\
\delta(t) = \, & \frac{(\mu(t)-r)^2}{\sigma^2(t)}, \nonumber 
\end{align}
where
\[h_V(t,u_x)=\exp(-\rho t)V\left( (V^\prime)^{-1}(u_x\exp(\rho t)) \right) -u_x(V^\prime)^{-1}(u_x\exp(\rho t)).\]
We are looking for the symmetry generated by the generating function $\Omega (t, x, u, u_x)$.
Hereafter, we assume that the function $h_V$ defined above is a smooth function in a suitable open subset of $\mathbb{R}^2$. 

\begin{theorem}\label{theorem:merton1}
The function $\Omega$ generates a contact symmetry of equation~\eqref{eq:main} if and only if it admits one of the following forms
\begin{align*}
\Omega_1 & = \exp \left( -\frac{r(r-1)t}{2} \right) [u\cdot u_x^r-x u_x^{r+1}] + G_1 (t,u_x), \\
\Omega_2 & =-u+ G_2 (t,u_x), 
\\
\Omega_3 & =\exp(r t) xu_x  + G_3
(t,u_x),\\
\Omega_4 &= G_4(t,u_x),
\end{align*}
where $G_1,G_2,G_3,G_4 \colon \mathbb{R}_+ \times \mathbb{R} \rightarrow \mathbb{R}$ are smooth functions satisfying the~PDEs
\begin{align}
2\exp\left(\frac{-r(r-1)t}{2}\right)u_x^{r} h_V+\delta(t) u_x^2 \partial_{u_xu_x}G_1+2\partial_tG_1 & = 0,\label{eq:K1}\\
2u_x\partial_{u_x}h_V-2 h_V+\delta(t) u_x^2  \partial_{u_xu_x}G_2+2\partial_tG_2 & = 0,\label{eq:K2}\\
2 \exp(r t) u_x \partial_{u_x}h_V+2 x r\exp(r t) u_x + \delta(t) u_x^2  \partial_{u_xu_x}G_3+2\partial_tG_3 & = 0,\label{eq:K3}\\
\delta(t)u_x^2  \partial_{u_xu_x}G_4+2\partial_tG_4 & = 0.\label{eq:K4}
\end{align}
\end{theorem}

Finally, Theorem~\ref{theorem:merton1} and Theorem~\ref{theorem_noether} allow us to obtain the explicit forms of the local martingales of Merton's model.

\begin{corollary} \label{cor:merton}
Let $U(t,x)$ be a classical solution to equation~\eqref{eq:main} and let $X_t$ be the solution to equation~\eqref{eq:mainMerton1} with $(\gamma,c)$ satisfying the equalities~\eqref{eq:cstar} and~\eqref{eq:gammastar}. Then, the processes
\begin{align*}
O_{1,t} & = \exp \left( -\frac{r(r-1)t}{2} \right) [U(t,X_t)\partial_xU(t,X_t)^r-\partial_xU(t,X_t)^{r+1}] + G_1 (t,\partial_xU(t,X_t)),  \\
O_{2,t} & =U(t,X_t)+ G_2 (t,\partial_xU(t,X_t)), 
\\
O_{3,t} & =\exp(rt)X_t \partial_xU(t,X_t)+ G_3
(t,\partial_xU(t,X_t)),\\
O_{4,t}&= G_4(t,\partial_xU(t,X_t)),
\end{align*} 
are local martingales.
\end{corollary}

\begin{proof}[Proof of Theorem \ref{theorem:merton1}]
The generating function $\Omega$ is a (contact) symmetry of the PDE if and
only if the following set of determining equations holds
\begin{align}
\frac{\delta}{2}  u_x^2 \partial_{u_x u_x} \Omega + u_x \partial_u \Omega \cdot
\partial_{u_x} K + \partial_x \Omega \cdot \partial_{u_x} K - \partial_u
\Omega \cdot K -\partial_{u_x} \Omega \cdot \partial_{x} K  + \partial_t \Omega & = 0, \label{eq:determining1}\\
{\delta} u_x^2 \partial_{u u_x} \Omega +  {\delta} u_x \partial_{u_x x} \Omega -
{\delta}  \partial_x \Omega & = 0,  \label{eq:determining2}\\
\frac{\delta}{2}  u_x^2 \partial_{u u} \Omega + \delta u_x \partial_{u x} \Omega +
\frac{\delta}{2} \partial_{x x} \Omega & = 0.  \label{eq:determining3}
\end{align}
We can differentiate equation~{\eqref{eq:determining2}} with respect to~$u$
and equation~{\eqref{eq:determining3}} with respect to $u_x$, and equate the
term $\partial_{u u u_x} \Omega$ to obtain
\begin{equation}
(4 \partial_{u u} \Omega + \partial_{u u_x x} \Omega) u_x^2 +
(\partial_{u_x x x} \Omega + 7 \partial_{u x} \Omega) u_x + 2
\partial_{x x} \Omega \label{eq:determining4} = 0.
\end{equation}
Differentiating equation~{\eqref{eq:determining2}} with respect to~$x$, we can
get an expression of $\partial_{u u_x x} \Omega$ in terms of
$\partial_{u_x x x} \Omega$ and~$\partial_{x x} \Omega$. 
Replacing now the
obtained expression in equation~{\eqref{eq:determining4}} yields
\begin{equation}
u_x \partial_{u u} \Omega + \partial_{u x} \Omega = 0
\label{eq:determining5} .
\end{equation}
If we differentiate equation~{\eqref{eq:determining2}} with respect to~$u$ and
use equation~{\eqref{eq:determining5}}, then we have
\begin{equation}
\partial_{u x} \Omega = 0. \label{eq:determining6}
\end{equation}
Inserting equation~{\eqref{eq:determining6}} in equation~{\eqref{eq:determining5}},
we get
\begin{equation*}
\partial_{u u} \Omega = 0, 
\end{equation*}
from which, thanks to equations~{\eqref{eq:determining3}} and~{\eqref{eq:determining6}}, we obtain
\[ \partial_{x x} \Omega = 0. \]
This means that ${\Omega}$ is a function of the form
\begin{equation}
\Omega (t, x, u, u_x) = f_1 (t, u_x) u + f_2 (t, u_x) x + f_3 (t, u_x) .
\label{eq:determining8}
\end{equation}
If we replace expression~{\eqref{eq:determining8}} inside the determining equations~{\eqref{eq:determining1}},~{\eqref{eq:determining2}}, and~{\eqref{eq:determining3}}, we have that $f_1$, $f_2$, and $f_3$ have to satisfy the
following set of equations
\begin{align}
u_x^2 \partial_{u_x u_x} f_1 + 2 \partial_t f_1 & = 0 ,
\label{eq:determining9}\\
u_x^2 \partial_{u_x u_x} f_2 + 2 \partial_t f_2 -2f_2r &= 0 ,
\label{eq:determining10}\\
- 2 u_x f_1 \cdot \partial_{u_x} K - 2 f_2 \cdot \partial_{u_x} K + 2 f_1
\cdot K +\delta u_x^2 \partial_{u_x u_x} f_3 + 2 \partial_t f_3 & = 0,
\label{eq:determining11}\\
u_x^2 \partial_{u_x} f_1 + u_x \partial_{u_x} f_2 - f_2 & = 0 .
\label{eq:determining12}
\end{align}
Solving equation~{\eqref{eq:determining12}} with respect to~$f_2$, we obtain
that
\begin{equation}
f_2 = -u_x f_1 + g_1 (t) u_x . \label{eq:determining13}
\end{equation}
Replacing the expression~{\eqref{eq:determining13}} in equation~{\eqref{eq:determining10}} and using equation~{\eqref{eq:determining9}}, we
have the equation
\[- 2 u_x^2 \partial_{u_x} f_1 + 2 u_x \partial_t g_1+2ru_xf_1-2rg_1u_x = 0, \]
from which we get that 
\begin{equation}
f_1 = \left[ d(t)+ \frac{b(t)}{r} \right] u_x^r-   \frac{b(t)}{r}, \label{eq:determining14} 
\end{equation}
where $b(t)=\partial_t g_1-rg_1$.
Replacing equation~{\eqref{eq:determining14}} in equation~{\eqref{eq:determining9}}, we obtain
\begin{align*}
r(r-1) \left[ d(t)+\frac{b(t)}{r} \right] +2 \left( d^\prime(t)+\frac{b^\prime(t)}{r} \right) & =0,\\
\partial_{t,t}g_1-r\partial_{t}g_1 & = 0,
\end{align*}
giving
\[ b(t)=c_2, \quad d(t) = d_1\exp \left( \frac{-r(r-1)t}{2} \right)-\frac{d_2}{r} ,  \quad g_1(t) =d_3  \exp(rt)-\frac{d_2}{r} , \]
for some arbitrary constants $d_1$, $d_2$, and~$d_3$.
Therefore, we have
\begin{equation*}
f_1 =d_1\exp\left(\frac{-r(r-1)t}{2}\right)u_x^r-\frac{d_2}{r}.
\end{equation*}
By~{\eqref{eq:determining13}}, we obtain
\begin{equation*}
f_2 =-d_1\exp \left( \frac{-r(r-1)t}{2} \right) u_x^{r+1}+ d_3 \exp(rt) u_x.
\end{equation*}
Inserting the previous expression of $f_1$, $f_2$, and $f_3$ in~\eqref{eq:determining11}, we get that $\Omega$ is a contact symmetry of equation~{\eqref{eq:main}} if and only if it is a linear combination of the following expressions
\begin{align*}
\Omega_1 & = \exp \left( -\frac{r(r-1)t}{2} \right) [u\cdot u_x^r-x u_x^{r+1}] + G_1 (t,u_x), & & d_1 = 1,d_2=d_3=0, \\
\Omega_2 & =-u+ G_2 (t,u_x), & &  d_2 =- r,d_1= d_3 = 0,
\\
\Omega_3 & =\exp(r t) xu_x  + G_3
(t,u_x), & &d_3 = 1,d_1=d_2=0,\\
\Omega_4 &= G_4(t,u_x), & & d_1 =d_2=d_3=0,
\end{align*}
where $G_1$, $G_2$, $ G_3$, and $G_4$ are smooth solutions to the PDEs satisfying equations~{\eqref{eq:K1}},~{\eqref{eq:K2}},~{\eqref{eq:K3}}, and~{\eqref{eq:K4}}.
\end{proof}

Equations~{\eqref{eq:K1}},~{\eqref{eq:K2}},~{\eqref{eq:K3}}, and~{\eqref{eq:K4}} can be solved
explicitly for some special form of~$K (t,x,u_x)$. 
Taking, in particular, the following two expressions 
\[K_1=h_1+rxu_x, \quad h_1=-\exp(-\rho t)[\log(u_x)+\rho t+1], \]
and 
\[K_2=h_2+rxu_x, \quad h_2= -\exp \left(\frac{\rho t}{\theta -1 }\right) u_x^{\frac{\theta}{\theta-1}}\frac{\theta -1}{\theta}, \] 
derived by taking the \emph{isoelastic utility} functions, also known as { \em constant relative risk aversion} utilities (see~\cite{pratt_risk_1978}) defined~as
\[ V(z)=\log(z) \qquad \text{and} \qquad V(z)=\frac{z^{\theta}}{\theta}, \quad \theta \in \mathbb{R}, \]
respectively. 
If we denote by $\Omega^1_1= \exp \left( -\frac{r(r-1)t}{2} \right) [u\cdot u_x^r-x u_x^{r+1}] + G_1^1 (t,u_x)$ and by $\Omega^2_1= \exp \left( -\frac{r(r-1)t}{2} \right) [u\cdot u_x^r-x u_x^{r+1}] + G_1^2 (t,u_x)$ the symmetries of the equation~\eqref{eq:main} when $K=K_1$ and $K=K_2$, respectively, we have that $G^1_1$ solves the equation
\begin{equation}
\Gamma_1(t,u_x) +\delta(t) u_x^2\partial_{u_x, u_x}G^1_1+2\partial_{t}G^1_1=0,
\label{eq:h1}
\end{equation}
where
\begin{equation}
\Gamma_1(t,u_x) = 2\exp \left( -\frac{r(r-1)t-\rho t}{2} \right) [1-2u_x^r-\log(u_x)u_x^r-\rho t u_x^r].
\label{eq:h2}
\end{equation}
Making the ansatz
\begin{equation*}
G^1_1(t,u_x)=\phi_1(t)u_x^r+\phi_2(t)u_x^r\log(u_x)+\phi_3(t),
\end{equation*}
we have that $G^1_1$ solves~\eqref{eq:h1} if and only if $\phi_1,\phi_2,\phi_3$ solve the following~ODEs
\begin{align*}
\phi_1^\prime= \, &2\exp \left( -\frac{r(r-1)t}{2}-\rho t \right) (2+\rho t) - \delta(t)r(r-1)\phi_1-\delta(t)(r-2)\phi_2-\delta(t)\phi_2, 
\\
\phi_2^\prime= \, &2\exp \left( -\frac{r(r-1)t}{2}-\rho t \right) -\delta(t) r(r-1)\phi_2,  
\\
\phi_3^\prime=\, & 2\exp \left( \frac{r(r-1)t}{2}-\rho t \right). 
\end{align*}
In the same way  $G^2_1$ solves 
\begin{equation}
\Gamma_2(t,u_x) + u_x^2\partial_{u_x, u_x}G^2_1+2\partial_{t}G^2_1=0,
\label{eq:G21}
\end{equation}
where
\begin{equation*}
\Gamma_2(t,u_x) = 2u_x^{\frac{\theta}{\theta-1}}\exp \left(-\frac{r(r-1)t}{2}+\frac{\rho}{\theta-1} t \right) \left[ 1-2\frac{\theta-1}{\theta}u_x^r \right] .
\end{equation*}
With the ansatz
\begin{equation*}
G^2_1(t,u_x)=\phi_1(t)u_x^{\frac{\theta}{\theta-1}}+\phi_2(t)u_x^{\frac{\theta}{\theta-1}+r-1},
\end{equation*}
equation~\eqref{eq:G21} holds if and only if $\phi_1$, $\phi_2$, and $\phi_3$ solve the following~ODEs
\begin{align*}
\phi_1^\prime =\, &\exp \left( -\frac{r(r-1)t}{2}+\frac{\rho}{\theta-1} t \right) -\frac{\delta(t)}{2} \left( \frac{\theta}{\theta -1} \right) \left( \frac{\theta}{\theta -1}-1 \right) \phi_1, 
\\
\phi_2^\prime= \, &4\frac{\theta -1}{\theta} \exp \left( -\frac{r(r-1)t}{2}+\frac{\rho}{\theta-1} t \right) -\delta(t) \left(\frac{\theta}{\theta -1}+r \right) \left( \frac{\theta}{\theta -1}+r-1 \right) \phi_2. 
\end{align*}
If we denote by $\Omega^1_2= -u+ G_2^1 (t,u_x)$ and by $\Omega^2_2= -u+ G_2^2 (t,u_x)$ the symmetries of the equation~\eqref{eq:main} when $K=K_1$ and $K=K_2$, respectively, then we get that $G_2^{i}$, $ i=1,2$, solve~\eqref{eq:K2} with $h_1$ and $h_2$ given by~\eqref{eq:h1} and~\eqref{eq:h2}. \\
With the ansatz
\begin{equation*}
G^1_2(t,u_x)=\phi_1(t)+ \phi_2(t)\log(u_x),
\end{equation*}
The function $G^1_2$ solves~\eqref{eq:K2} (with $h=h_1$) if and only if $\phi_1$ and $\phi_2$ solve the following~ODEs
\begin{align*}
\phi_1'(t)=&\frac{\delta(t)}{2}\phi_2(t) - \exp(-\rho t)-\exp(-\rho t)(\rho t+1) 
\\
\phi_2'(t)=&-\exp(-\rho t).
\end{align*}
With the ansatz
\begin{equation*}
G^2_2(t,u_x)=\phi_1(t)u_x^{\frac{\theta}{\theta-1}},
\end{equation*}
the function  $G^2_2$ solves~\eqref{eq:K2} (with $h=h_2$) if and only if $\phi_1$ solves the following~ODE
\begin{equation*}
\phi'_1(t)=-\frac{\delta(t)\theta}{2(\theta-1)^2}\phi_1-\frac{1}{\theta}\exp \left( \frac{\rho}{\theta-1} t\right).
\end{equation*}
If we denote by $\Omega^1_3= \exp(rt)xu_x+ G_3^1(t,u_x)$ and by $\Omega^2_3= \exp(rt)xu_x+ G_3^2 (t,u_x)$ the symmetries of the equation~\eqref{eq:main} when $K=K_1$ and $K=K_2$, respectively, then $G_3^{i}$, $ i=1,2$, solve~\eqref{eq:K3} with $h_1$ and $h_2$ given above, respectively. \\
With the ansatz
\begin{equation*}
G^1_3(t,u_x)=\phi_1(t)+ \phi_2(t)\log(u_x),
\end{equation*}
the function $G^1_3$ solves~\eqref{eq:K2} (with $h=h_1$) if and only if $\phi_1$ and $\phi_2$ solve the following~ODEs
\begin{align*}
\phi_1'(t)=&\frac{\delta(t)}{2}\phi_2(t) - \exp((r-\rho) t),\\ 
\phi_2'(t)=&0. 
\end{align*}
With the ansatz
\begin{equation*}
G^2_3(t,u_x)=\phi_1(t)u_x^{\frac{\theta}{\theta-1}},
\end{equation*}
The function $G^2_3$ solves~\eqref{eq:K3} (with $h=h_2$) if and only if $\phi_1$ solves the following~ODE
\begin{equation*}
\phi'_1(t)=-\frac{\delta(t)\theta}{2(\theta-1)^2}\phi_1+\exp \left(\left( r+\frac{\rho}{\theta-1} \right) t\right) .
\end{equation*}

\subsection{Non-Markovian case}\label{subsection:nonmarkovian}

We consider here the case where $\mu(t)$ and $\sigma(t)$ are predictable continuous stochastic processes with respect to the filtration generated by~$\mathcal{F}_t$, that is, the problem now fits in the more general model treated in Section~\ref{s:stoch-hjb}.
This case is relevant, for example, when we are considering stochastic volatility models (see, e.g.,~\cite{benth_merton_2003,fouque_portfolio_2017,lorig_portfolio_2016} for stochastic volatility models and~\cite{Oksendal2016} for the non-Markovian Merton problem of the form approached here).
We assume also that $g(x,\omega)$ is a $\mathcal{F}_t$ random field. In this case, the value function is a random field depending on the time $t$ and the variable $x$ of the form 
\[U(t,x)=\mathbb{E}\left[\left.\int_t^T L(s,\alpha_t) \mathd s+ g(X_T,\omega)\right|\mathcal{F}_t\cap\{X_t=x\}\right].\] 
The random field $U$ satisfies the following backward stochastic~PDE
\begin{equation}\label{eq:stochasticMerton}
\mathd U(t,x)+\sup_{(c,\gamma)\in K}\mathcal{H}^S(t,x,\nabla U(t,x), D^2U(t,x),\nabla \Psi(t,x) ,(c,\gamma))\mathd t=\Psi(t,x)\mathd W_t
\end{equation}
where 
\[\mathcal{H}^S(t,x,u_x, u_{xx},\psi_x ,(c,\gamma))=\exp(-\rho t)V(c)+(\gamma(\mu(t)-r)+r)xu_x-c u_x+x\sigma(t)\gamma\psi_x+\frac{1}{2}u_{xx} \sigma(t)^2 \gamma^2 x^2. \]
The optimal value of the function $(c, \gamma)$ is given by the solution to the system
\begin{align*}
\partial_{c}\mathcal{H}=\exp(-\rho t)V^\prime(c)-u_x & =0,\\
\partial_{\gamma}\mathcal{H}=(\mu(t)-r)xu_x+x\sigma(t)\psi_x+u_{xx}\sigma^2(t) x^2 \gamma  &=0 ,
\end{align*}
which means that 
\begin{equation}
\gamma^*=-\frac{(\mu(t)-r)u_x+\sigma(t)\psi_x}{xu_{xx}\sigma(t)^2},\label{eq:gammastar2}
\end{equation}
while $c^*$ is given by equation \eqref{eq:cstar}.
This implies that 
\[
H^S(t,x,u_x,u_{xx},\psi_x)=\frac{((\mu(t)-r)u_x+\sigma(t)\psi_x)^2}{2\sigma(t)^2u_{xx}}+K(t,x,u_x),\]
where $K(t,x,u_x)$ is given by equation \eqref{eq:MertonK}. 
In the following, we write 
\[\delta^S(t)=\frac{(\mu(t)-r)^2}{\sigma(t)^2}, \]
where we recall that here $\mu$ and $\sigma$ are generic predictable continuous stochastic processes. 
So we consider a generator function $\Omega^S(t,u,u_x,\omega)$, depending explicitly on~$\omega$.

\begin{theorem}\label{theorem:merton2}
The generator function $\Omega^S(t,u,u_x,\omega)$ is a symmetry of equation~\eqref{eq:stochasticMerton} in the sense of Definition~\ref{def:contact-symmetry-for-dU} if and only if $\Omega^S$ has one of the following forms
\begin{align*}
\Omega_1 & = \exp \left( -\frac{r(r-1)t}{2} \right) [u\cdot u_x^r-x u_x^{r+1}] + G_1 (t,u_x), \\
\Omega_2 & =-u+ G_2 (t,u_x), 
\\
\Omega_3 & =\exp(r t) xu_x  + G_3
(t,u_x),\\
\Omega_4 &= G_4(t,u_x),
\end{align*}
where $G_1^S,G_2^S,G_3^S \colon \mathbb{R}_+ \times \mathbb{R}  \times \Omega \rightarrow \mathbb{R}$ are smooth predictable random fields satisfying the following random~PDEs
\begin{align}
2\exp\left(-\frac{r(r-1)t}{2}\right)u_x^{r} h_V+\delta^S(t) u_x^2 \partial_{u_xu_x}G_1+2\partial_tG_1 & = 0, \label{eq:K1S}\\
2u_x\partial_{u_x}h_V-2 h_V+\delta^S(t)u_x^2  \partial_{u_xu_x}G_2+2\partial_tG_2 & = 0, \label{eq:K2S}\\
2 \exp(r t) u_x \partial_{u_x}h_V+\delta^S(t)u_x^2  \partial_{u_xu_x}G_3+2\partial_tG_3 &=0,\label{eq:K3S}\\
\delta^S(t)u_x^2  \partial_{u_xu_x}G_4+2\partial_tG_4&=0.\label{eq:K4S}
\end{align}
\end{theorem}
\begin{proof}
Since 
\[H^S(t,x,u_x,u_{xx},0)=\frac{\delta^S(t)}{2u_{xx}}+K(t,x,u_x),\] 
which is formally equal to $H$ defined in Subsection \ref{subsection:Markovian}, the theorem can be easily proven using the same argument exploited in the proof of Theorem \ref{theorem:merton1}.
\end{proof}

\begin{remark}
The symmetries $\Omega^S_i$ of Theorem \ref{theorem:merton2} depend on $\omega\in \mathcal{W}$ since the functions $G_i$ solve the random equations \eqref{eq:K1S}, \eqref{eq:K2S}, \eqref{eq:K3S}, and \eqref{eq:K4S} (where the random dependence is given by $\delta^S(t)$).
\end{remark}

\begin{corollary}
Let $(U(t,x),\Psi(t,x))$ be a classical solution to equation \eqref{eq:stochasticMerton} and let $X_t$ be the solution to equation \eqref{eq:mainMerton1} with $(\gamma,c)$ satisfying equalities \eqref{eq:cstar} and \eqref{eq:gammastar2}. 
Then, the processes 
\begin{align*}
\tilde{O}_{1,t} & = \exp \left( -\frac{r(r-1)t}{2} \right) [U(t,X_t)\partial_xU(t,X_t)^r-\partial_xU(t,X_t)^{r+1}] + G_1 (t,\partial_xU(t,X_t))-I_1(t,U,\nabla U, \nabla \Psi),   \\
\tilde{O}_{2,t} & =U(t,X_t)+ G_2 (t,\partial_xU(t,X_t))-I_2(t,U,\nabla U, \nabla \Psi), 
\\
\tilde{O}_{3,t} & =\exp(rt)X_t \partial_xU(t,X_t)+ G_3
(t,\partial_xU(t,X_t))-I_3(t,U,\nabla U, \nabla \Psi),\\
\tilde{O}_{4,t}&= G_4(t,\partial_xU(t,X_t))-I_4(t,U,\nabla U, \nabla \Psi), 
\end{align*} 
are local martingales. 
Here, $I_1$, $I_2$, $I_3$ and $I_4$ are the integral expressions associated with $O_{1,t}$, $O_{2,t}$, $O_{3,t}$ and $O_{4,t}$, respectively, by the relation given in equation~\eqref{eq:tildeO}.
\end{corollary}
\begin{proof}
The first statement follows from Theorem \ref{theorem:merton2} and Theorem \ref{theorem_noether2}. 
\end{proof}

In the particular case where $r=0$, $V(z)=\frac{z^{\theta}}{\theta}$ (where $\theta \in \mathbb{R}$) or without the consumption $V(z)=0$, $c=0$, we can obtain the following stronger result.

\begin{corollary}\label{cor:stoch-merton}
Suppose that $r=0$ and $V(z)=\frac{z^{\theta}}{\theta}$. Then, we have that 
\[O_t= - U(t,X_t)-\frac{1}{\theta}X_t \partial_xU(t,X_t)\]
is a local martingale. Furthermore, if $V=0$ (and we consider $c=0$) we have that, for any $c_1,c_2 \in \mathbb{R}$,
\[O^{c_1,c_2}_t=c_1 U(t,X_t)+c_2 X_t \partial_x U(t,X_t) \]
is a local martingale.
\end{corollary}
\begin{proof}
If $V(z)=\frac{z^{\theta}}{\theta}$ we have $ h_V(t,u_x) =-\exp(-\rho t)u_x^{\frac{\theta}{\theta-1}}\frac{\theta -1}{\theta}$. 
This implies that $\left(\frac{1}{\theta}-1\right)\partial_{u_x}h_V-h_V=0$. So, using equation \eqref{eq:K2S} and \eqref{eq:K3S}, we get that 
\[\Omega_5=\Omega_2-\frac{1}{\theta}\Omega_3= - u-\frac{1}{\theta}xu_x+G_5(t,u_x),\]
where $G_5(t,u_x)$ is any solution to the equation 
\begin{equation}\label{eq:G5}
\delta^S(t)u_x^2 \partial_{u_xu_x}G_5+2\partial_tG_5=0,
\end{equation}
is a symmetry of the equation \eqref{eq:stochasticMerton}. 
A particular solution to equation~\eqref{eq:G5} is $G_5 \equiv 0$, in which case $\Omega_5$ has the form $\Omega_5=-u-\frac{xu_x}{\theta}$. But $-u-\frac{xu_x}{\theta}$ satisfies the hypotheses of Corollary \ref{cor:noether2}, from which we get the thesis. The second part of the corollary can be proven in a similar way.
\end{proof}

As already mentioned in the introduction, the construction of the martingales obtained in Corollaries~\ref{cor:merton} and~\ref{cor:stoch-merton} could be deeply connected to the well-known explicit solutions of Merton's optimal portfolio problem (see, e.g.,~\cite{rogers_optimal_2013} for a review and~\cite{benth_merton_2003,biagini_robust_2017} for recent developments on the explicit solutions of Merton's problem). 
The investigation of the link between these two notions will be the subject of a future paper.

\paragraph{Acknowledgments.}
The first, second, and fourth author are funded by Istituto Nazionale di Alta Matematica ``Francesco Severi'' (INdAM), Gruppo Nazionale per l'Analisi Matematica,
la Probabilità e le loro Applicazioni (GNAMPA): ``Lie's Symmetries Analysis of Stochastic Optimal Control Problems with Applications''.
The first and third author are funded by the DFG under Germany's Excellence Strategy - GZ 2047/1, project-id 390685813.

\bibliographystyle{abbrv}
\bibliography{Mertbib.bib}

\end{document}